\documentclass[fleqn]{article}

\usepackage{./myCommands}
\usepackage{amsthm}
\usepackage{amscd}
\usepackage{tikz-cd}
\newlength{\bibitemsep}
\newlength{\bibparskip}
\let\oldthebibliography\thebibliography
\renewcommand\thebibliography[1]{%
  \oldthebibliography{#1}%
  \setlength{\parskip}{\bibitemsep}%
  \setlength{\itemsep}{\bibparskip}%
}

\usepackage{hyperref} 
\usepackage[left=1.5in, top=1.0in, right=1.5in, bottom=1.0in]{geometry}

\togglefalse{colcom}

\newcommand{\singlespacing}{%
  \let\CS=\small\renewcommand{\baselinestretch}{1.0}\CS}
\newcommand{\doublespacing}{%
  \let\CS=\small\renewcommand{\baselinestretch}{1.6}\CS}

\newtheorem{thm}{Theorem}[subsection]
\newtheorem{lem}[thm]{Lemma}

\newtheorem{cor}[thm]{Corollary}

\theoremstyle{definition}
\newtheorem{defn}[thm]{Definition}

\newtheorem{eg}{Example}[section]

\newtheorem{assumption}{Assumption}[section]

\theoremstyle{remark}
\newtheorem{rem}[thm]{Remark}

\numberwithin{equation}{section}

\newcommand{\ev}{\mathcal{E}}

\newcommand{\wY}{\widehat{Y}}
\newcommand{\wh}[1]{\widehat{#1}}

\newcommand{\wEps}{\widetilde{\epsilon}}

\newcommand{\decomp}{\mathcal{P}}

\newcommand{\tr}{\,\mathrm{tr}\,}

\newcommand{\Cov}{\mathrm{Cov}}

\newcommand{\edge}[1]{\set{#1}}

\begin{document}
\author{Michael P. Casey}
\title{Bulk Johnson-Lindenstrauss Lemmas}
\maketitle

\begin{abstract}
For a set $X$ of $N$ points in $\R^D$,
the Johnson-Lindenstrauss lemma provides random linear maps that approximately preserve all pairwise distances in $X$ -- up to multiplicative error $(1\pm \epsilon)$ with high probability -- using a target dimension of $O(\epsilon^{-2}\log(N))$. 
Certain known point sets actually require a target dimension this large -- any smaller dimension forces at least one distance to be stretched or compressed too much. 
What happens to the remaining distances? 
If we only allow a fraction $\eta$ of the distances to be distorted beyond tolerance $(1\pm \epsilon)$, 
we show a target dimension of $O(\epsilon^{-2}\log(4e/\eta)\log(N)/R)$ is sufficient for the remaining distances.
With the stable rank of a matrix $A$ as $\norm{A}_F^2/\norm{A}^2$, 
the parameter $R$ is the minimal stable rank over certain $\log(N)$ sized subsets of $X-X$ or their unit normalized versions, involving each point of $X$ exactly once. 
The linear maps may be taken as random matrices with $\iid$ zero-mean unit-variance sub-gaussian entries. 
When the data is sampled $\iid$ as a given random vector $\xi$, refined statements are provided; the most improvement happens when $\xi$ or the unit normalized $\wh{\xi-\xi'}$ is isotropic, with $\xi'$ an independent copy of $\xi$, and includes the case of $\iid$ coordinates. 
\end{abstract}

\begin{itemize}
\item e-mail: mpcasey@alumni.duke.edu
\item MSC: primary 68Q87, 68R12, 60B20; secondary 62G30, 68T09 
\item keywords: dimension reduction, Johnson-Lindenstrauss lemma, Hanson-Wright inequality, stable rank, effective rank, intrinsic dimension, order statistics, Walecki construction, bulk, batch, minibatch, random projection
\end{itemize}


\section{Introduction}
The Johnson-Lindenstrauss lemma~\cite{JohnsonLindenstrauss1984} concerns the approximate preservation of distances in a finite point set in Euclidean space. 
Specifically, for a subset $X\subset\R^D$ of $N$ points and a tolerance $\epsilon\in(0,1)$, there exist $k\times D$ matrices $Z$ and a constant $\gamma(\epsilon)$ for which 
\begin{equation}\label{eqn:JLUnsquaredDistancesPairwise}
(1-\epsilon)\norm{x-x'}_2\leq \sqrt{\frac{\gamma(\epsilon)}{k}}\norm{Z(x-x')}_2 \leq (1+\epsilon)\norm{x-x'}_2\tag{JL}
\end{equation}
holds \emph{for all} pairs of points $x,x'\in X$ simultaneously, with probability at least $1-\delta$, provided 
\[
k=O(\epsilon^{-2}\log(N^2/\delta))=:D_{JL}(N)=:D_{JL}. 
\]
The matrices $Z$ are drawn randomly, and much work has been done since the original paper to equip $Z$ with special properties, such as allowing fast matrix multiplication, preserving sparsity, restricting the matrix entries to discrete distributions, and so forth; see~\cite{NelsonReview2020} for a recent review. 
The matrix $Z$ provides a linear method for dimension reduction, which, at the very least, reduces the amount of space needed to store the dataset $X$ on the computer, provided one can work with approximate versions of the pairwise distances. 
One would expect that \qt{robust} downstream algorithms that depend on distance data, now working on the pointset $ZX$, should still return answers that approximate those found on the original pointset $X$, but now in shorter time, with less memory needed, etc. 
People appreciate this lemma, in theory~\cite{VempalaRandomProjection2004}, for the above reasons, and if the algorithm scales linearly in the ambient dimension, in time or in space, then processing $ZX$ instead of $X$ will yield a proportional improvement. 

However, certain algorithms, including many nearest-neighbor algorithms, scale exponentially in the dimension, especially if they only use space linear in $N$, due to packing arguments. 
For comparison to brute force, the all pairs nearest-neighbor problem may already be solved in $O(DN^2)$ time by scanning the points of $X$ with respect to their distance to each query $x$. 
If the algorithm scales like $DN\exp(D)$, this scaling is only an improvement for $D=\log(N)$, and even then one would really prefer $D=o(\log N)$, as $N^2$ is too expensive when $N$ is large. 
Consequently, if we use multiplication by $Z$ as a preprocessing step, the target dimension $D_{JL}=O(\epsilon^{-2}\log(N^2))$ is much too large to be practical, as $\exp(D_{JL})$ is now polynomial in $N$ with exponent at least $2/\epsilon^2$ with $\epsilon<1$. 

Apart from searching for better algorithms, it is then natural to ask whether there are matrices $Z$ with target dimension $k$ much smaller than $D_{JL}$ that would still satisfy equation~\eqref{eqn:JLUnsquaredDistancesPairwise} \emph{for all} pairs of points of $X$. 
However, Larsen and Nelson~\cite{LarsenNelsonOptimal2014} showed no such matrix exists for general datasets -- a union of the standard simplex and a sufficient number of Gaussian vectors forces at least one distance to be stretched or compressed too much. 
On the other hand, if further assumptions are made on the pointset, smaller $k$ is possible, for instance, when $X$ already lies in a low-dimensional subspace~\cite{SarlosImproved2006} or when its unit difference set $\wh{X-X}$ has low Gaussian complexity~\cite{KlartagMendelson2005}, even while allowing many more points in the dataset. 

If one considers a given algorithm \qt{robust}, one would hope that failing to preserve a few distances should not markedly change the final output; though, one would ideally have to prove such behavior for that algorithm. 
One is then led to ask: what happens to the other distances between the points of $ZX$ when $k$ is smaller than $D_{JL}$?
To be concrete, can we approximately preserve $(1-\eta)$ of all the distances, for some fraction $\eta\in(0, 1/2)$, ideally with $k=o(D_{JL}(N))$? 
The results in this paper show this is possible when $\eta$ is not too small and the data has high or even moderate \qt{intrinsic dimension}, in a sense to be defined later. 
As a preview, for data sampled $\iid$ like a random vector $\xi\in \R^D$, corollary~\ref{cor:ArbitraryIIDCoordsBulkJL} shows that if $\xi$ has $\iid$ coordinates (No moment assumptions are made.), we can take 
\[
k=\frac{C'}{\epsilon^2}\left(\log(4e/\eta)\log(6De/\zeta)+\frac{\log(N^2/\delta)}{\eta D}\right)
\]
for preserving $(1-\eta)(1-\zeta)$ of all pairwise distances, with probability at least $1-2\delta$ over the draw for $Z$ and $X$, provided $N=\Omega(\zeta^{-1}D\log(N/\delta))$ and $N=\Omega(D\log(6De/\zeta))$. 
The matrix $Z$ may have $\iid$ standard Gaussian, or in general, zero-mean unit-variance sub-gausssian coordinates. 
This estimate for $k$ is an \qt{improvement} over the $D_{JL}$ target dimension as soon as $\eta D=\omega(1)$ or say a fractional power of $\log(N)$; \qt{improvement} must be in quotes, as we have only guaranteed \qt{most} distances are approximately preserved, not all. 
For general $\xi$, the $1/\eta D$ is replaced by a $1/\eta \wh{r}$, with $\wh{r}$ the \emph{intrinsic dimension} $1/\norm{\E\wh{y}\wh{y}^\top}$ of the unit normalized random vector $\wh{y}=\wh{\xi-\xi'}$ for an independent copy $\xi'$ of $\xi$. 
We have $1\leq \wh{r}\leq D$ always, and we may estimate it using corollary~\ref{cor:EstimatingRHat}. 

Both of our main results -- theorems~\ref{thm:FreeDecompBulkJL} for general datasets and theorem~\ref{thm:ArbitraryOnUnitSphereBulkJL} for $\iid$ samples -- rely on a dual viewpoint for the dimension reduction problem, namely, instead of asking how $Z$ transforms the data $X$, we ask how $X^\top$ transforms the \emph{test matrix} $Z^\top$; we can then exploit known properties of how general matrices act on $Z^\top$ or its columns. 
Standard results like the Hanson-Wright inequality may be viewed in this light, and we indeed do so in this paper. 
Treating $Z^\top$ as a test matrix with known properties has been done previously in randomized numerical linear algebra~\cite{HalkoMartinssonTropp2010}; however, unlike there, slow decay in singular values is actually a good case for us. 

The rest of the paper is organized as follows. 
The main argument allowing us to quantify how many distances can be preserved is in section~\ref{sec:OrderStatControl} and leads to our first theorem~\ref{thm:FreeDecompBulkJL}, which, by itself, only suggests smaller target dimensions $k$ may be possible by considering \qt{small} batches of difference vectors. 
We then recall the Walecki construction in section~\ref{sec:Walecki}, which gives us a way to choose these batches that works well for $\iid$ samples as well as the standard simplex. 
Section~\ref{sec:IID} then presents the rest of our results specializing to data sampled $\iid$ from a given distribution, which may just be a larger dataset. 
This section leads up to our second main theorem~\ref{thm:ArbitraryOnUnitSphereBulkJL} allowing us to make $k$ depend on the intrinsic dimension $\wh{r}$ mentioned above. 
We also show how to estimate $\wh{r}$ from the data. 
The appendix contains proofs for the properties of $Z$ that we use in the paper, namely, a variant of the Hanson-Wright inequality, and a corresponding independent proof in the Gaussian case in order to have decent explicit constants for $k$. 

\subsection{Notation}
Suppose $X\subset \R^D$ is a point set of size $N$. 
Given an arbitrary ordering of the points of $X$, 
let $Y$ be the set of difference vectors
\[
Y:=\set{x_i-x_j\st x_i,x_j\in X,\,0\leq i<j\leq N-1} 
\]
and $\wY$ be their unit normalized versions
\[
\wY:=\set{y/\norm{y}_2\st y\in Y}\subset S^{D-1}. 
\]
The number of elements of a set $\Upsilon$ is denoted by $\abs{\Upsilon}$. 
We set $(x,y)=x^\top y$ for the usual Euclidean inner product, with $\norm{x}_2^2 = (x,x)$. 
We also denote by $o$ the origin $(0,\ldots, 0)$ in any $\R^j$. 

Let $A$ be a $D\times M$ matrix, which we write as $A\in \R^{D\times M}$. 
From~\cite[page~76]{GolubVanLoan2013}, the singular value decomposition (SVD) for $A$ is the factorization $A=U\Sigma V^\top$ with $U\in \R^{D\times D}$ and $V\in \R^{M\times M}$ orthogonal 
and 
\[
\Sigma=\diag(\sigma_1, \ldots, \sigma_p) \qtext{with} p=\min\set{D,\, M}
\]
for $\sigma_1\geq \sigma_2\geq \ldots\geq \sigma_p\geq 0$. 
Let $\vec{\sigma}=(\sigma_1,\ldots, \sigma_p)$ as a vector in $\R^p$. 
We write $\norm{A}=\norm{\vec{\sigma}}_\infty =\sigma_1$ for the operator norm of $A$, while $\norm{A}_F=\norm{\vec{\sigma}}_2$ is the Frobenius (or Hilbert-Schmidt) norm of $A$. 
We may also compute $\norm{A}_F$ via
\[
\norm{A}_F^2 = \sum_i \norm{A_i}_2^2 = \sum_{i,j} A_{ij}^2 
\]
where the $A_i$ may be taken as the rows or the columns of $A$. 
Vectors are treated as column vectors unless otherwise indicated. 

The $\infty$-stable rank $r_\infty(A)$ of a matrix $A$ is the ratio
\[
r_\infty(A):=\frac{\norm{A}_F^2}{\norm{A}^2}
=\frac{\norm{\vec{\sigma}}_2^2}{\norm{\vec{\sigma}}_\infty^2}
=\frac{\sum \sigma_i^2}{\sigma_1^2}. 
\]
There are other variants of stable rank~\cite{NaorYoussefRestricted2017}, but only the $4$-stable rank $r_4(A)$ will make an appearance in this paper:
\[
r_4(A):=\frac{\norm{A}_F^4}{\norm{AA^\top}_F^2}
=\frac{\norm{\vec{\sigma}}_2^4}{\norm{\vec{\sigma}}_4^4}
=\frac{\big(\sum\sigma_i^2\big)^2}{\sum \sigma_i^4}
\geq \frac{\big(\sum\sigma_i^2\big)^2}{\sigma_1^2\sum \sigma_i^2}=r_\infty(A)
\]
So $r_4(A)\geq r_\infty(A)$ always, and both of these are at most $p$, the rank of $A$. 

We make use of the $\psi_2$-norm and $\psi_1$-norm, defined for a random variable $\omega$ as
\[
\norm{\omega}_{\psi_2}:=\inf\set{t>0\st \E \exp(\omega^2/t^2)\leq 2}, 
\qtext{and}
\norm{\omega}_{\psi_1}:=\inf\set{t>0\st \E \exp(\abs{\omega}/t)\leq 2}. 
\]
See~\cite[section~2.5 and 2.7]{VershyninHDP2018}. 

\subsubsection{Isotropic Random Vectors and the Intrinsic Dimension}\label{sec:IsotropicIntrinsic}
A particular condition on the second moment matrix of a random vector will be useful in this paper. 
\begin{defn}
A random vector $\xi\in\R^D$ is \emph{isotropic} if $\Sigma(\xi):=\E\xi\xi^\top =\Id_D$. 
\end{defn}
A working example is any vector $\xi$ with $\iid$ zero-mean unit-variance coordinates. 
Basic properties of isotropic random vectors include $\E \norm{\xi}_2^2=D$ via the cyclic property of the trace, and that $\xi$ isotropic is equivalent to $\E(\xi,y)^2=\norm{y}_2^2$ for any $y\in \R^D$. See~\cite[page~43-5]{VershyninHDP2018} for more on such vectors.

We can assign a version of $\infty$-stable rank to the distribution of a random vector via the \emph{intrinsic dimension}. 
\begin{defn}
The \emph{intrinsic dimension} $r(\xi)$ of a random vector $\xi\in \R^D$ is the ratio \[
r(\xi):=\frac{\tr\Sigma(\xi)}{\norm{\Sigma(\xi)}}
=\frac{\tr{\E\xi\xi^\top}}{\norm{\Sigma(\xi)}}
=\frac{\E\norm{\xi}_2^2}{\norm{\E\xi\xi^\top}}, 
\qtext{and for $c\neq 0$, }
r(c\,\xi)=\frac{c^2\,\E\norm{\xi}_2^2}{c^2\norm{\E\xi\xi^\top}}=r(\xi). 
\]
\end{defn}
Like the stable ranks, the intrinsic dimension of a vector in $\R^D$ is bounded by the ambient dimension, $D$, 
so isotropic random vectors realize the highest possible intrinsic dimension. 
In the literature, $r(\xi)$ is sometimes called the \emph{effective rank} of the second moment matrix $\Sigma(\xi)$, and is the stable rank of the matrix $\Sigma^{1/2}(\xi)$. 

Isotropic random vectors behave well under orthogonal projection. 
\begin{lem}\label{lem:UnitaryPreservesIsotropic}
Let $\Phi$ be a $d\times D$ matrix with orthonormal rows and $\xi\in \R^D$ an isotropic random vector. 
Then $\Phi(\xi)\in \R^d$ is also isotropic. 
\end{lem}
\begin{proof}
By linearity of expectation: 
$\E(\Phi \xi)(\Phi \xi)^\top = \Phi(\E \xi\xi^\top)\Phi^\top =\Phi \Id_D\Phi^\top = \Id_d$. 
\end{proof}
Note if $\Phi$ is scaled by constant $c\neq 0$, the intrinsic dimension is unchanged: $d=r(\Phi(\xi))=r(c\,\Phi(\xi))$. 


\section{Controlling Order Statistics}\label{sec:OrderStatControl}
In this paper, we only study dimension reduction matrices $Z:\R^D\to\R^k$ with isotropic rows, that is, 
$\E(Z_i, y)^2=\norm{y}_2^2$ 
for all $y$ and all rows $Z_i$.
So, $\E\norm{Zy}_2^2 = k\norm{y}_2^2$. 
We say more about isotropic random vectors in section~\ref{sec:IsotropicIntrinsic}. 
To guarantee equation~\eqref{eqn:JLUnsquaredDistancesPairwise} holds for a difference vector $y\in Y$, the usual proof for the Johnson-Lindenstrauss lemma considers each vector individually, providing upper bounds for the failure probabilities
\[
p_+(y):=\Prob_Z\set{\norm{Zy}_2^2  > (1+\wEps_+)k \norm{y}_2^2} 
\]
and 
\[
p_-(y):=\Prob_Z\set{\norm{Zy}_2^2  < (1-\wEps_-)k \norm{y}_2^2}, 
\]
implicitly working with the formulation
\begin{equation}\label{eqn:JLSquaredDistancesPairwise}
(1-\wEps_-)k\norm{x-x'}_2^2
\leq \norm{Z(x-x')}_2^2 
\leq (1+\wEps_+)k\norm{x-x'}_2^2\tag{$JL^2$}
\end{equation}
which is often easier to manage. 
In lemma~\ref{lem:EpsilonAdjustment} of the appendix, we show how to choose $\wEps_\pm$ in line with the original equation~\eqref{eqn:JLUnsquaredDistancesPairwise}. 

\tcr{Give the next lemma here or in the appendix AND mention that we treat $\epsilon$ in the HW proofs as the the appropriate $\widetilde{\epsilon}_\pm$ everywhere?}
\tcb{How bad is it to replace $\epsilon$ by $\wEps_\pm$ everywhere? Will they be eliminated in the theorem statements?}

Suppose $\wEps_-\leq\wEps_+$, as it will for this paper.
If the distribution of $Z$ is sufficiently nice, sub-gaussian for example, then one may show $p_++p_-\leq 2\exp(-\wEps_-^2 k/C)$ for each fixed $y$, with $C$ a constant depending on the distribution of $Z$. 
By the union bound (Boole's inequality), the probability that a random $Z$ fails for some $y\in Y$ is at most 
\[
\binom{N}{2}2\exp(-\wEps_-^2 k/C)\leq N^2\exp(-\wEps_-^2 k/C)<\delta, 
\]
provided $k>C\wEps_-^{-2}\log(N^2/\delta)=D_{JL}$. 
The probability there is some $Z$ that succeeds for all $y\in Y$ is then at least $1-\delta$, so at least one such $Z$ exists. 

The argument above considers the vectors $y\in Y$ one at a time, making sure $Z$ succeeds for each $y$; if we consider the unit normalized vectors $\wh{y}\in \wY$, we may view this argument as controlling the extreme order statistics of the random variables $\norm{Z\wh{y}_i}_2^2$, induced by $Z$, namely 
\[
(1-\wEps_-)k\leq \norm{Z\wh{y}_{(0)}}_2^2 \leq \ldots \leq \norm{Z\wh{y}_{(j)}}_2^2
\leq \ldots \leq \norm{Z\wh{y}_{\big(\binom{N}{2}-1\big)}}_2^2\leq (1+\wEps_+)k. 
\]
If we only wish to preserve a fraction of the distances, say $(1-\eta)$ with $\eta$ hopefully small, we can consider controlling the intermediate or central order statistics of the $\norm{Z\wh{y}_i}_2^2$ instead. 
We do so as follows. 

If we divide the difference vectors into batches of size $M$, and preserve $(1-\eta)M$ distances there, then we still recover 
\[
(1-\eta)M\binom{N}{2}/M=(1-\eta)\binom{N}{2} \qtext{
of the total distances.} 
\]
We assume $\eta M$ is a strictly positive integer here, and for simplicity of discussion, we also assume $M$ divides $N$; we shall return to this point later. 
Let $\Upsilon\subset Y$ be a given batch of size $M$. 
Each given matrix $Z$ also induces an ordering on the points of $\Upsilon$: with $\wh{y}=y/\norm{y}_2$, 
\[
\norm{Z\wh{y}_{(0)}}_2^2\leq\ldots\leq \norm{Z\wh{y}_{(i)}}_2^2
\leq \ldots \leq \norm{Z\wh{y}_{(M-1)}}_2^2. 
\]
As $Z$ is random, this ordering is too, treating $Y$ as fixed. 
If we could guarantee that $\norm{Z\wh{y}_{((1-\eta)M)}}_2^2\leq (1+\epsilon)k$, then \emph{all} $\norm{Z\wh{y}_{(i)}}_2^2$ with $i\leq (1-\eta)M$ also have this upper bound, with an analogous statement for a lower bound of $(1-\epsilon)k$ on $\norm{Z\wh{y}_{(\eta M-1)}}_2^2$, altogether guaranteeing $(1-2\eta)M+2>(1-2\eta)M$ of the vectors of $\Upsilon$ have 
\[
(1-\wEps_-)k\leq \norm{Z\wh{y}}_2^2\leq (1+\wEps_+)k. 
\]

We need to control the probabilities 
\begin{align}\label{eqn:OrderStatEventBatchLevel}
p_+(\Upsilon):&=\Prob_Z\set{\norm{Z\wh{y}_{((1-\eta)M)}}_2^2>(1+\wEps_+)k}
\intertext{and} 
p_-(\Upsilon):&=\Prob_Z\set{\norm{Z\wh{y}_{(\eta M-1)}}_2^2< (1-\wEps_-)k}. 
\end{align}
We can recast control of $p_\pm(\Upsilon)$ using the following lemma. 
Let $\wh{\Upsilon}=\set{\wh{y}_0, \ldots, \wh{y}_{M-1}}$, that is, the unit normalized version of $\Upsilon$. 

\begin{lem}\label{lem:OrderStatViaBlas3}
Let $Z$ be a random $k\times D$ random matrix. With the notation above, and $\Upsilon(\eta)$ running through all $\eta M$-sized subsets of $\Upsilon$ 
\[
p_+(\Upsilon)
\leq \max_{\Upsilon(\eta)}\min_{\Lambda_{\Upsilon(\eta)}}\binom{M}{\eta M}
    \Prob_Z\set{\norm{Z\Upsilon(\eta)\Lambda_{\Upsilon(\eta)}}_F^2>(1+\wEps_+)k\norm{\Upsilon(\eta)\Lambda_{\Upsilon(\eta)}}_F^2}
\]
and
\[
p_-(\Upsilon)
\leq\max_{\Upsilon(\eta)}\min_{\Lambda_{\Upsilon(\eta)}}\binom{M}{\eta M}
    \Prob_Z\set{\norm{Z\Upsilon(\eta)\Lambda_{\Upsilon(\eta)}}_F^2<(1-\wEps_-)k\norm{\Upsilon(\eta)\Lambda_{\Upsilon(\eta)}}_F^2} 
\]
with $\Lambda_{\Upsilon(\eta)}$ a diagonal matrix with strictly positive entries, chosen for each subset $\Upsilon(\eta)$.
\end{lem}
\begin{rem}
For simplicity, the rest of the paper will take $\Upsilon(\eta)\Lambda_{\Upsilon(\eta)}$ as  $\Upsilon(\eta)$ itself, or its unit normalized version $\wh{\Upsilon}(\eta)$; though, there may potentially be some improvement gained by the freedom in choosing each $\Lambda_{\Upsilon(\eta)}$. 
\end{rem}
\begin{proof}
For $p_+(\Upsilon)$, let $\wh{y}_\ast=\wh{y}_{((1-\eta)M)}$. 
If $\norm{Z\wh{y}_\ast}_2^2> (1+\wEps_+)k$, then $\wh{y}_\ast$ is part of a subset $\Upsilon(\eta)\subset \Upsilon$ of size $\eta M$ for which all the $Z\wh{y}$ have norms too large. 
For any given subset $\Upsilon(\eta)$, consider the following probabilities, with $\lambda_y>0$ chosen for each $y$, 
\begin{align*}
\Prob_Z\set{\ev(\Upsilon(\eta))}&:=\Prob_Z\set{\norm{Z\wh{y}}_2^2>(1+\wEps_+)k \txt{for all} y\in \Upsilon(\eta)}\\
&=\Prob_Z\set{\norm{Zy}_2^2>(1+\wEps_+)k\norm{y}_2^2 
    \txt{for all} y\in \Upsilon(\eta)}\\ 
&=\Prob_Z\set{\norm{Zy\lambda_y}_2^2>(1+\wEps_+)k\norm{y\lambda_y}_2^2 
    \txt{for all} y\in \Upsilon(\eta)}\\
&\leq\Prob_Z\set{\sum_{y\in \Upsilon(\eta)} \norm{Zy\lambda_y}_2^2>(1+\wEps_+)k\sum_{y\in \Upsilon(\eta)} \norm{y\lambda_y}_2^2}. 
\end{align*}
The second and third lines follow by the linearity of $Z$. 
Passing to the sum above allows an important change of viewpoint using the Frobenius norm, as each $Zy\lambda_y$ is a column of $Z(\Upsilon(\eta)\Lambda_{\Upsilon(\eta)})$, with $\Lambda_{\Upsilon(\eta)}$ the appropriate diagonal matrix:
\[
\Prob_Z\set{\ev(\Upsilon(\eta))}
    \leq \Prob_Z\set{\norm{Z\Upsilon(\eta)\Lambda_{\Upsilon(\eta)}}_F^2
            >(1+\wEps_+)k \norm{\Upsilon(\eta)\Lambda_{\Upsilon(\eta)}}_F^2}. 
\]

Now, there are $\binom{M}{\eta M}$ subsets $\Upsilon(\eta)$ of $\Upsilon$ of size $\eta M$, so a union bound over such subsets gives
\begin{align*}
p_+(\Upsilon)
&\leq \binom{M}{\eta M}\max_{\Upsilon(\eta)}\Prob_Z\set{\ev(\Upsilon(\eta))}\\
&=\binom{M}{\eta M}\Prob_Z\set{\norm{Z\Upsilon(\eta)\Lambda_{\Upsilon(\eta)}}_F^2
            >(1+\wEps_+)k \norm{\Upsilon(\eta)\Lambda_{\Upsilon(\eta)}}_F^2}.
\end{align*}
The argument for $p_-(\Upsilon)$ is similar, noting that 
$\norm{Z\wh{y}_{(\eta M-1)}}_2^2<(1-\wEps_-)k$ forces $\eta M$ vectors $Z\wh{y}$ to have squared norms too small, so their corresponding sum is too small as well. 
\end{proof}

We now make assumptions on the matrix $Z$, allowing us to make use of the stable rank of the minibatches $\Upsilon(\eta)$. 
\begin{cor}\label{cor:pUpsilonBoundGaussianCase}
With the notation as in lemma~\ref{lem:OrderStatViaBlas3}, 
$\wEps_+= \wEps_-\sqrt{2}$, and $\wEps_-\in(0,1)$, if $Z$ has $\iid$ standard Gaussian entries, then 
\[
\max\set{p_+(\Upsilon),\,p_-(\Upsilon)}\leq \binom{M}{\eta M}\max_{\Upsilon(\eta)}\exp\left(-\frac{k\wEps_-^2}{4}r_\infty(\Upsilon(\eta))\right). 
\]
One may replace $\Upsilon(\eta)$ by $\wh{\Upsilon}(\eta)$ on the right hand side. 
\end{cor}
One point we want to stress even now is the presence of the \emph{product} $k\,r_\infty$ in the exponent. 
If one wants a fixed failure probability, $k$ need not be as large when $r_\infty$ is sizable. 
We shall give several examples in this paper where $r_\infty$ is large. 
\begin{proof}
The proof follows immediately from lemma~\ref{lem:DirectBulkHWGaussianCase} in the appendix, with $A=\Upsilon(\eta)$, recalling $r_4\geq r_\infty$ for the $p_-(\Upsilon)$ case. 
With $C_+=8$ and $C_-=4$, we improve the denominator from 8 to 4 by setting
$C_+/\wEps_+^2=C_-/\wEps_-^2$ as in in lemma~\ref{lem:EpsilonAdjustment}. 
\end{proof}

Consider the term $\norm{Z\Upsilon(\eta)}_F^2$ in lemma~\ref{lem:OrderStatViaBlas3}, taking $\Lambda_{\Upsilon(\eta)}$ as the identity. 
(The following discussion will also hold for other scalings, such as $\wh{\Upsilon}(\eta)$.) 
Labeling the \emph{rows} of $Z$ as $Z_i$, we can interchange the sums implicit in $\norm{Z\Upsilon(\eta)}_F^2$ to our advantage:
\begin{align*}
\sum_{y\in \Upsilon(\eta)} \norm{Zy}_2^2
=\sum_{y\in \Upsilon(\eta)}\sum_{i=1}^k (Z_i, y)^2
=\sum_{i=1}^k\sum_{y\in \Upsilon(\eta)} (y, Z_i)^2
=\sum_{i=1}^k \norm{\Upsilon(\eta)^\top Z_i}_2^2. 
\end{align*}
Because we assume the rows $Z_i$ are isotropic, $\E\norm{\Upsilon(\eta)^\top Z_i}_2^2=\norm{\Upsilon(\eta)^\top}_F^2=\norm{\Upsilon(\eta)}_F^2$, 
and we have transformed the bound on $p_+(\Upsilon)$ to involve the probabilities
\[
\Prob_Z\set{\sum_{i=1}^k \norm{\Upsilon(\eta)^\top Z_i}_2^2 
>(1+\wEps_+)\sum_{i=1}^k \norm{\Upsilon(\eta)}_F^2}, 
\]
and $\Upsilon(\eta)^\top$ is now viewed as a fixed matrix \emph{acting on} the random vectors $Z_i$. 
We may thus take a dual viewpoint on the dimension reduction problem: instead of considering how the matrix $Z$ acts on each difference vector, we consider how the transposed batch of difference vectors $\Upsilon(\eta)^\top$ acts on the matrix $Z^\top$, as mentioned in the introduction. 
If we take the $Z_i$ to be independent, with $\iid$ mean-zero unit-variance sub-gaussian entries, we can use lemma~\ref{lem:HWstableranks} in the appendix to harness \emph{both} the independence of the $Z_i$ \emph{and} the Hanson-Wright inequality:

\begin{lem}\label{lem:BulkHWOrderStat}
With the notation as in lemma~\ref{lem:OrderStatViaBlas3} and $\wEps_-\leq \wEps_+$, let $Z$ be a $k\times D$ random matrix with $\iid$ mean-zero unit-variance sub-gaussian entries, then 
\[
p_\pm(\Upsilon)
\leq 
\binom{M}{\eta M}\max_{\Upsilon(\eta)}
\exp\left(-ck\min\left(\frac{\wEps_-^2 r_4(\Upsilon(\eta))}{K^4}, 
\frac{\wEps_-\, r_\infty(\Upsilon(\eta))}{K^2}\right)\right), 
\]
with $K=\norm{Z_{11}}_{\psi_2}$. 
One may replace $\Upsilon(\eta)$ by $\wh{\Upsilon}(\eta)$ on the right hand side. 
\end{lem}
\begin{proof}
To bound the probabilities on the right hand side of lemma~\ref{lem:OrderStatViaBlas3}, just take $A=\Upsilon(\eta)$ in lemma~\ref{lem:HWstableranks}. 
\end{proof}
\begin{rem}
Recall $r_4\geq r_\infty$ always, so if $K\geq 1$, we can write
\[
p_\pm(\Upsilon)
\leq 
\binom{M}{\eta M}\max_{\Upsilon(\eta)}
\exp\left(-C k\min(\wEps_-^2, \epsilon)
r_\infty(\Upsilon(\eta))\right). 
\]
\end{rem}

We now can control the probabilities in equations~\eqref{eqn:OrderStatEventBatchLevel} for a given batch $\Upsilon$ of size $M$. 
The control is in terms of $r_\infty(\Upsilon(\eta))$ or $r_4(\Upsilon(\eta))$ for subsets of size $\eta M$. 
Because the target dimension $k$ is a global parameter, it needs to be in terms of global quantitities, but the stable ranks above vary over minibatches. 
To make this transition and to help with bookkeeping, recall that a \emph{decomposition} of a graph is a partition of its edges. 

Let $\decomp$ be a decomposition of the complete graph on $N$ vertices, into batches $\Upsilon$ of size $M$. 
If $M$ does not divide $N$, that is, with $N=jM+n$, we can expand several of the batches to $M+s\geq M$ with $s=\ceil{n/j}<M$. 
For those batches, $\eta(M+s)$ need not be an integer, so take $\tilde{\eta}$ as
\begin{equation}\label{eqn:EtaTilde}
\tilde{\eta}=\eta\frac{M}{M+s} \qtext{in order for} \tilde{\eta}(M+s)=\eta M, 
\end{equation}
and set $\Upsilon(\eta)=\Upsilon(\tilde{\eta})$ when the batch size is not $M$.
Note smaller $\eta$ values only help us ensure a total fraction $(1-\eta)$ of distances is preserved.
For this decomposition, let 
\[
R_\infty(\eta M):=
R_\infty(\eta M; \decomp):=\min_{\Upsilon\in\decomp}\min_{\Upsilon(\eta)\subset \Upsilon} r_\infty(\Upsilon(\eta))
\]
be the minimum stable rank of the $\eta M$ sized minibatches from such batches.
We then have our first theorem, written in terms of $\epsilon$ in the original equation~\eqref{eqn:JLUnsquaredDistancesPairwise}.  
\begin{thm}\label{thm:FreeDecompBulkJL}
Let $0<\eta<1$, with $\eta M\in \N$, and let $0<\epsilon\leq 2/3$.  
For a set $X$ of $N$ points in $\R^D$, $R_\infty(\eta M)$ as above, and $Z$ a $k\times D$ matrix with $\iid$ mean-zero unit-variance sub-gaussian entries, 
equation~\eqref{eqn:JLUnsquaredDistancesPairwise} holds
for at least $(1-2\eta)\binom{N}{2}$ pairs $x,x'\in X$, with probability at least $1-\delta$, provided
\[
k\geq C\frac{\max(K^4,K^2)}{\epsilon^2}\left(\log(2e/\eta)\frac{\eta M}{R_\infty(\eta M)}
+\frac{\log(N^2/(M\delta))}{R_\infty(\eta M)}\right), 
\qtext{and} \gamma(\epsilon)=1+\epsilon^2. 
\]
Here, $K=\norm{Z_{11}}_{\psi_2}$ and $C$ is an absolute constant. 
In the case of standard Gaussian entries, one can replace $C\max(K^4,K^2)$ and $\gamma(\epsilon)$ by 
\[
C(\epsilon):=4\left(\frac{(1+\epsilon)^2+(1-\epsilon)^2\sqrt{2}}{4}\right)^2 < 2.25  
\qtext{and} 
\gamma(\epsilon)=\frac{(1+\epsilon)^2 + (1-\epsilon)^2\sqrt{2}}{1+\sqrt{2}}, 
\]
respectively. 
When the distribution for $Z$ is understood, we denote $C\max(K^4,K^2)$ or $C(\epsilon)$ by 
$C_{\lbrack\ref{thm:FreeDecompBulkJL}\rbrack}$, and likewise $\gamma(\epsilon)$ by $\gamma_{\lbrack\ref{thm:FreeDecompBulkJL}\rbrack}$. 
\end{thm}
To make sense of the above, suppose $\eta M=D$, and then recall $1\leq R_\infty(\eta M)\leq D$ as the stable rank is always bounded above by the ambient dimension.   
If $R_\infty(\eta M)=c D$, for $c$ bounded away from 0, the term attached to $\log(2e/\eta)$ becomes constant, while the remaining becomes constant as soon as 
$D\gtrsim \log(2N^2/(D\delta))$. 

Note $R_\infty(\eta M)$ depends on the decomposition $\decomp$, which we are free to choose at will. 
The remainder of the paper will address how to choose $\decomp$ (and the batch size $M$). 
\begin{proof}
We treat the upper and lower bounds on $\norm{Z(x-x')}_2^2$ separately. 
If 
\begin{equation}\label{eqn:unionBound}
\sum_{\Upsilon\in\decomp}(p_+(\Upsilon)+p_-(\Upsilon))\leq \delta, 
\end{equation}
then with probability at least $1-\delta$ none of the events defining $p_\pm(\Upsilon)$ hold, over all the batches $\Upsilon$ of the decomposition. 
So by equation~\eqref{eqn:OrderStatEventBatchLevel}, 
with probability at least $1-\delta$, we preserve at least (recalling $\tilde{\eta}<\eta$ when needed)  
\[
\sum_{\Upsilon\in\decomp}(1-\eta)\abs{\Upsilon}=(1-\eta)\binom{N}{2}
\]
of the squared distances within a $(1+\wEps_+)$ tolerance, and another $(1-\eta)\binom{N}{2}$ of the squared distances within a $(1-\wEps_-)$ tolerance. 
By the pigeonhole principle, at least $(1-2\eta)\binom{N}{2}$ of the distances are approximatley preserved on both sides. 

It remains to choose $k$ to ensure equation~\eqref{eqn:unionBound} holds. 
We always have, by lemma~\ref{lem:binomEstimate}, 
\[
\binom{M+s}{\tilde{\eta}(M+s)}\leq \left(\frac{e(M+s)}{\eta M}\right)^{\eta M}
=\exp\left(\log\left(\frac{e(M+s)}{\eta M}\right)\eta M\right)
<\exp(\log(2e/\eta)\eta M), 
\]
while $\binom{M}{\eta M}\leq \exp(\log(e/\eta)\eta M)$, 
so by lemma~\ref{lem:BulkHWOrderStat}
\[
p_+(\Upsilon)+p_-(\Upsilon)\leq 2\exp\left(-\frac{c}{\max(K^4,K^2)} k\wEps_-^2 R_\infty(\eta M) + \log(2e/\eta)\eta M\right). 
\]
There are at most $\floor{\binom{N}{2}/M}\leq N^2/(2M)$ batches $\Upsilon$ in the decomposition, expanding several of the batches to absorb any remainder $N\mod M$ if necessary, so we need $k$ to satisfy
\[
2(N^2/(2M))\exp\left(-\frac{c}{\max(K^4,K^2)} k\wEps_-^2 R_\infty(\eta M) 
    + \log(2e/\eta)\eta M\right)\leq\delta. 
\]

Using lemma~\ref{lem:EpsilonAdjustment} with $C_+=C_-=1$, we achieve the desired bound for $k$ in terms of $\epsilon$ by taking
\[
C=\frac1{c}\left(\frac{1+\epsilon^2}{2}\right)^2 < 0.522\frac1{c}. 
\]
To replace $C\max(K^4, K^2)$ in the Gaussian case, use corollary~\ref{cor:pUpsilonBoundGaussianCase} with lemma~\ref{lem:EpsilonAdjustment} on $C_+=2C_-$ to find 
\[
\frac{4}{\wEps_-^2}=\frac{C(\epsilon)}{\epsilon^2}. \qedhere
\]
\end{proof}


From lemma~\ref{lem:OrderStatViaBlas3} we are free to scale vectors in each batch $\Upsilon$ to have unit norm. 
So we can also define for a given decomposition
\[
\wh{R}_\infty(\eta M; \decomp)=\min_{\wh{\Upsilon}\in\decomp}\min_{\wh{\Upsilon}(\eta)\subset \wh{\Upsilon}} r_\infty(\wh{\Upsilon}(\eta)). 
\]
This normalization allows us to use linear algebra to control $r_\infty(\wh{\Upsilon}(\eta))$ deterministically, for any of the $\eta M$-sized subsets of $\wh{\Upsilon}$, once we have control of $\sigma_1(\wh{\Upsilon})$. 
When the underlying data is random, we can thus avoid the need to take union bounds over the minibatches. 
\begin{lem}\label{lem:StableRanksOfSubsets}
Let $A$ be a $D\times M$ matrix with columns of constant norm. 
If $B$ is a $D\times m$ submatrix of $A$, then
\[
r_\infty(B)\geq \max\left(\frac{m}{M}r_\infty(A), 1\right). 
\]
In particular, if $r_\infty(A)\geq cM$, then $r_\infty(B)\geq \max(cm,1)$. 
\end{lem}
To be useful, we need $c\gg 1/m$ here.
\begin{proof}
To control $\norm{B}$, partition the matrix $A$ as $A=(B'|B)$ with $B'$ a $D\times (M-m)$ matrix. 
Viewing the unit sphere $S^{m-1}$ as $o\times S^{m-1}$ in $S^{M-1}$, 
\[
\norm{A} = \max_{x\in S^{M-1}}\norm{(B'|B)x}_2\geq \max_{x\in S^{m-1}}\norm{(B'|B)x}_2=
\max_{x\in S^{m-1}}\norm{Bx}_2=\norm{B}. 
\]
That is, $\norm{B}\leq \norm{A}$. 
Note $r_\infty(A)=r_\infty(cA)$ for any nonzero scalar $c$, so we may assume the columns of $A$ all have norm 1. 
Under this assumption, 
\[
r_\infty(B)=\frac{\sum\norm{B_j}_2^2}{\norm{B}^2}
=\frac{m}{M}\frac{\sum\norm{A_j}_2^2}{\norm{B}^2}
\geq \frac{m}{M}\frac{\sum\norm{A_j}_2^2}{\norm{A}^2}=\frac{m}{M}r_\infty(A). 
\]
Recalling $r_\infty\geq 1$ always completes the proof. 
\end{proof}

If we set 
\[
\wh{R}_\infty(M):=\wh{R}_\infty(M;\decomp)
:=\min_{\Upsilon\in\decomp} r_\infty(\wh{\Upsilon}), 
\qtext{we can state the following theorem.}
\]
\begin{thm}\label{thm:UnitFreeDecompBulkJL}
Let $0<\eta<1$, with $\eta M\in \N$, and $0<\epsilon\leq 2/3$. 
For a set $X$ of $N$ points in $\R^D$, $\wh{R}_\infty(M)$ as above, and $Z$ a $k\times D$ matrix with $\iid$ mean-zero unit-variance sub-gaussian entries, 
then equation~\eqref{eqn:JLSquaredDistancesPairwise} holds
for at least $(1-2\eta)\binom{N}{2}$ pairs $x,x'\in X$, with probability at least $1-\delta$, provided
\[
k\geq \frac{C_{\lbrack\ref{thm:FreeDecompBulkJL}\rbrack}}{\epsilon^2}\left(\log(2e/\eta)
    \frac{M}{\wh{R}_\infty(M)}
+\frac{\log(N^2/(M\delta))}{\max(\eta\wh{R}_\infty(M), 1)}\right) 
\qtext{and}
\gamma(\epsilon)=\gamma_{\lbrack\ref{thm:FreeDecompBulkJL}\rbrack}. 
\]
Here, $C_{\lbrack\ref{thm:FreeDecompBulkJL}\rbrack}$ depends on $K=\norm{Z_{11}}_{\psi_2}$. 
In the case of independent standard Gaussian entries, $C_{\lbrack\ref{thm:FreeDecompBulkJL}\rbrack}<2.25$.  
\end{thm}
\begin{proof}
In the proof of theorem~\ref{thm:FreeDecompBulkJL}, we shall replace lemma~\ref{lem:BulkHWOrderStat} by 
\[
p_\pm(\wh{\Upsilon})
\leq 
\binom{M}{\eta M}
\exp\left(-ck \eta r_\infty(\wh{\Upsilon})\min\left(\frac{\wEps_-^2}{K^4}, 
\frac{\wEps_-}{K^2}\right)\right). 
\]
By lemma~\ref{lem:StableRanksOfSubsets}, $r_\infty(\wh{\Upsilon}(\eta))\geq \eta r_\infty(\wh{\Upsilon})$, no matter the subset $\wh{\Upsilon}(\eta)$ chosen. 
We can then upper bound the right hand side of lemma~\ref{lem:BulkHWOrderStat}, recalling $r_4\geq r_\infty$. 
\end{proof}

The choice of decomposition $\decomp$ matters, yielding very different $\wh{R}_\infty(M)$ values, even with $M$ fixed, as seen in the following. 
\begin{eg}\label{eg:BatchChoiceMattersStableRankSimplex}
The difference vectors for the (vertices of) the standard simplex in $\R^D$ are $\set{e_i-e_j}_{i\neq j}$, with $1\leq i\leq D$. 
Suppose the decomposition $\decomp$ involved \qt{stars} made by subsets $\set{e_j - e_1}$ for $2\leq j\leq M+1\leq D$. 
Using these difference vectors as rows, the corresponding matrix $\wh{\Upsilon}$ looks like (relabeling as necessary)
\[
\wh{\Upsilon}=\frac1{\sqrt{2}}\begin{pmatrix}
\Id_{M}& -\I
\end{pmatrix}
\]
with $\I=(1,\ldots, 1)^\top\in \R^{M}$. 
If $z=(0,\ldots, 0, 1)^\top$, we have $\norm{\wh{\Upsilon}}^2\geq \norm{\wh{\Upsilon} z}_2^2 = M/2$, so $\wh{R}_\infty(M)\leq r_\infty(\wh{\Upsilon})\leq M/(M/2)=2$. 
Because $\wh{R}_\infty(M)$ is bounded by a constant, this decomposition is of no use in theorem~\ref{thm:UnitFreeDecompBulkJL}. 

If we instead consider a decomposition involving \qt{cycles} of length $D=M$, that is, subsets 
\[
\set{e_i-e_j\st i,j\txt{appear exactly twice}}, 
\]
then the corresponding difference vectors form a circulant matrix
\[
\Upsilon=\begin{pmatrix}
1 &  0 & \cdots & 0 &-1\\
-1 & 1 &        &   & 0\\ 
  & -1 & \ddots &   & \vdots\\
  &    & \ddots & 1 & 0\\
  &    &        & -1 & 1
\end{pmatrix}. 
\]
Viewing $\Upsilon$ as a complex matrix, we can diagonalize it using the discrete Fourier transform matrix $V=F_D$ as $V^{-1}\Upsilon V=\diag(\lambda_1,\ldots, \lambda_D)$. See~\cite[page~222]{GolubVanLoan2013}.  
Here, 
\[
\lambda_j = \left(\bar{F}_D\begin{pmatrix}
1 \\
-1 \\
o
\end{pmatrix}\right)_j
=1-\omega^{j-1}
\qtext{with}
\omega=\exp(2\pi i/D). 
\]
The squares of the singular values of $\Upsilon$ are then the eigenvalues of $\Upsilon^\top \Upsilon$, that is, 
\[
\sigma_j^2(\Upsilon)=\lambda_j^\ast \lambda_j = (1-\bar{\omega}^{j-1})(1-\omega^{j-1})
=2-2\cos(2\pi(j-1)/D)\leq 4, 
\]
with equality achieved when $D$ is even at $j=1+D/2$. 
Because the cycle has length $D$ here, we have $r_\infty(\Upsilon)\geq (2D)/4 = D/2$, for each such cycle. 

\emph{If} such a decomposition involved only such cycles, we could conclude (because the vectors have constant norm) $\wh{R}_\infty(M)\geq D/2$, which would be very useful for theorem~\ref{thm:UnitFreeDecompBulkJL}. 
\end{eg}

We review in the next section a construction originally due to Walecki that provides such cycle decompositions as above when $D$ is odd, and the next best thing when $D$ is even. 
The construction will also be useful when our data set is drawn $\iid$; in particular, we can address cases where the minimal stable ranks $R_\infty$ or $\wh{R}_\infty$ are too pessimistic, but \qt{most} batches have larger stable ranks. 


\section{The Walecki Construction}\label{sec:Walecki}
The sets $Y$ and $\wY$ describe $\binom{N}{2}$ directions in space.
Even if the data generating these directions is sampled independently, the directions themselves are not independent; for example, $x-y$ is not independent of $x-z$, while $x-y$ and $z-w$ are, assuming $x,y,z$, and $w$ are drawn independently.  
However, we are partitioning the directions into batches. 
If we can guarantee that in each batch, the directions are independent or only \qt{weakly} dependent, we can exploit these properties, ensuring the stable ranks of many batches are bounded below by given values. 

Viewing $Y$ or $\wY$ as corresponding to the edges of the complete graph $K_N$ on $N$ vertices, we are asking for a partition of the edges, a decomposition, such that each subset involves each vertex once, or at most twice (say). 
Thankfully, Walecki provided such a construction in the 1880's; 
see~\cite[page~161, Sixi\`eme R\'ecr\'eation]{LucasRecreations1882} for the original French and \cite{AlspachWonderfulWalecki2008} for an English overview. 
We use a different indexing scheme than presented in~\cite{AlspachWonderfulWalecki2008}, which is easier for us to use. 

We can arrange $N$ vertices in the complex plane as follows: $N-1$ vertices corresponding to the $(N-1)$st roots of unity, and the remaining vertex at the origin $o$. 
Let $\omega=\exp(2\pi i/(N-1))$. 
We can thus label the nonzero vertices as $\set{\omega^j}_{j=0}^{N-2}$. 
We partition their corresponding edges based on their products $W_p:=\set{\edge{\omega^j, \omega^k} \st \omega^j\omega^k=\omega^p}$, or more formally: 
\[
W_p:=\set{\edge{\omega^j,\omega^k}\st 0\leq j\neq k\leq N-2 \txt{and} j+k \equiv p \mod N-1}
\]
The $N-1$ sets $\set{W_p}$ represent \emph{1-regular subgraphs} of $K_N$: each vertex has degree one, for if $\edge{\omega^j,\omega^k}$ and $\edge{\omega^j,\omega^l}$ are in $W_p$, then
\[
j+k \equiv p \equiv j+l, \qtext{that is,} k\equiv p-j\equiv l, 
\]
forcing $k=l$ because $0\leq k,l\leq N-2$. 
There are only $N-1$ vertices on the circle, so there are at most $\floor{(N-1)/2}$ edges in $W_p$. 

The cyclic group $\Z_{N-1}$ generated by $\omega$ acts freely on the vertices of the circle via counterclockwise rotations, so $\Z_{N-1}$ also acts on the $W_p$ sets via
\[
\omega W_p := \set{ \edge{\omega \omega^j, \omega \omega^k}\st j+k \equiv p \mod N-1} 
= W_{(p+2)\mod N-1}, 
\]
corresponding to the product $(\omega\omega^j)(\omega\omega^k)=\omega^2\omega^{j+k}$. 
Consequently, it is enough to discuss $W_0$ and $W_1$. 

With $k>0$, the edges of $W_0$ are of the form $\edge{\omega^k, \omega^{-k}}$, while those of $W_1$ are of the form $\edge{\omega^k, \omega^{-(k-1)}}$. 
To each edge $\edge{\omega^k, \omega^{-k}}$ in $W_0$, there is a corresponding edge in $W_1$, namely $\edge{\omega^k, \omega^{-(k-1)}}$, so $\abs{W_1}\geq \abs{W_0}$ and $W_1$ includes the edge $\edge{\omega, 1}$. 
When $N-1$ is odd, the only nonzero vertex on the real line is 1; when $N-1$ is even, both vertices $1$ and $-1$ are present. 
Consequently, $\abs{W_0}$ is determined by the number of vertices strictly in the upper half plane:
\[
\abs{W_0}=\begin{cases}
(N-2)/2 \txt{if $N-1$ is odd},\\
(N-3)/2 \txt{if $N-1$ is even}. 
\end{cases}
\]

When $N-1$ is odd, 
recall $\abs{W_1}\leq \floor{(N-1)/2}$ because $W_1$ is 1-regular, so by the above, $\abs{W_1}=\abs{W_0}$, with $W_1$ avoiding the vertex $\omega^{-(N-2)/2}=\omega^{N/2}$, as its left most edge is 
$\edge{\omega^{(N-2)/2}, \omega^{-(((N-2)/2)-1)}}$. 
Form the augmented sets $\tilde{W}_0=W_0\cup\edge{o, 1}$ and $\tilde{W}_1=W_1\cup\edge{\omega^{-(N-2)/2}, o}$; each $\tilde{W}_i$ is a \emph{1-factor}, as it is 1-regular and spans all $N$ vertices. 
We can thus form a cycle using $\tilde{W}_0$ and $\tilde{W}_1$: explicitly, in cycle notation, 
\[
(o, 1, \omega^1, \omega^{-1}, \ldots, \omega^j, \omega^{-j}, \ldots, \omega^{(N-2)/2}, \omega^{-(N-2)/2})
\]
which has length $2(N-2)/2+2=N$ as it should for covering all the $N$ vertices.

When $N-1$ is even, $\abs{W_0}=(N-3)/2 = \floor{(N-1)/2}-1$, so $W_1$ can contain at most one additional edge; it does, via $(-1,\, \omega^{-(k_\ast-1)})$ with $k_\ast=(N-1)/2$. 
If $N\geq 7$, there are at least two edges in $W_0$, so 
split $W_0$ into $W_0^+$ and $W_0^-$, corresponding to those vertices with nonnegative and strictly negative real parts respectively. 
When $\abs{W_0}=(N-3)/2$ is even, that is, $N\equiv 3\mod 4$, both $W_0^\pm$ are of the same size $(N-3)/4$, while in the other case, we have $\abs{W_0^+}=\ceil{(N-3)/4}$ and $\abs{W_0^-}=\floor{(N-3)/4}$. 
Form the augmented sets $\tilde{W}_0^+=W_0^+\cup\edge{o,1}$ and $\tilde{W}_0^-=W_0^-\cup\edge{o, -1}$; these sets are 1-regular, of sizes
\[
\abs{\tilde{W}_0^\pm}=1+\abs{W_0^\pm}=1+(N-3)/4=(N+1)/4. 
\qtext{when} N\equiv 3\mod 4, 
\]
while when $N\equiv 1\mod 4$, that is, when $(N-3)/2$ is odd, 
\[
\abs{\tilde{W}_0^-}=1+\abs{W_0^-}=1+\frac1{2}\left(\frac{N-3}{2}-1\right)
=\frac{N-1}{4}
\]
and
\[
\abs{\tilde{W}_0^+}=1+\abs{W_0^+}=1+\frac1{2}\left(\frac{N-3}{2}+1\right)
=\frac{N+3}{4}. 
\]
The sets $\tilde{W}_0^+, \tilde{W}_0^-, W_1$ now form the cycle  
\[
(o, 1, \omega^1, \omega^{-1}, \ldots, \omega^j, \omega^{-j}, \ldots, \omega^{(N-3)/2}, \omega^{-(N-3)/2}, -1), 
\]
which has length $2(N-3)/2 + 3=N$, again as it should. 

Extending the $\Z_{N-1}$ group action to send the origin $o$ to itself, 
we can thus form $\floor{(N-1)/2}$ cycles $\mathcal{C}_j$ of length $N$ using the above, recalling there are $N-1$ different $W_p$ sets. 
Explicitly, 
\begin{equation}\label{eqn:splitCycles}
\mathcal{C}_j=\begin{cases}
\tilde{W}_{2j}^+\coprod \tilde{W}_{2j}^-\coprod W_{2j+1} &\txt{if $N-1$ is even}\\
\tilde{W}_{2j}\coprod \tilde{W}_{2j+1} &\txt{if $N-1$ is odd}. 
\end{cases}
\end{equation}
When $N-1$ is even, all $\binom{N}{2}$ edges are covered, while when $N-1$ is odd, $\tilde{W}_{N-2}$ still remains, but is still a 1-factor. 

We thus have, considering the parity of $N$ instead of $N-1$, 
\begin{lem}[Walecki Construction]\label{lem:WaleckiConstruction}
The complete graph $K_N$ has a decomposition into $(N-1)/2$ cycles of length $N$ when $N$ is odd, and $(N-2)/2$ cycles of length $N$ and a 1-factor when $N$ is even. 
\end{lem}

For reference later, we also record
\begin{cor}\label{cor:WaleckiSplitCycles}
Consider the cycles in lemma~\ref{lem:WaleckiConstruction}, corresponding to equation~\eqref{eqn:splitCycles}. 
When $N$ is even, these cycles split as a pair of 1-factors of size $N/2$. 
When $N$ is odd, the cycles split as three 1-regular subgraphs when $N\geq 7$. 
When $N\equiv 3 \mod 4$, their sizes are
\[
\abs{W_{2j+1}}=(N-1)/2 \qtext{and} \abs{\tilde{W}_{2j}^\pm}=(N+1)/4, 
\]
while when $N\equiv 1 \mod 4$, their sizes are
\[
\abs{W_{2j+1}}=(N-1)/2, \quad 
\abs{\tilde{W}_{2j}^-}=(N-1)/4, \qtext{and}
\abs{\tilde{W}_{2j}^+}=(N+3)/4. 
\]
\end{cor}

Returning to the sets of difference vectors $y_{ij}=x_i-x_j\in Y$ with $X\subset\R^D$ a set of $N$ points, we can assign each $y_{ij}$ to a unique 1-regular subgraph by the correspondence $y_{ij}\leftrightarrow \edge{\omega^i, \omega^j}$ when $0<i<j$ and $y_{0j}\leftrightarrow\edge{o, \omega^j}$ for $0<j$.  
Most useful for us is the following lemma; note we shall be considering batches of size (at least) $M$ drawn from within these subgraphs. 
\begin{lem}\label{lem:IndependentWithSplitWalecki}
Let $X\subset \R^D$ be a set of $N$ points drawn $\iid$ from a common distribution. 
With the correspondence above, the vectors within each subgraph from corollary~\ref{cor:WaleckiSplitCycles} and lemma~\ref{lem:WaleckiConstruction} are independent, as are their unit norm versions. 
When $N$ is even, there are $N-1$ subgraphs involved. 
When $N\geq 7$ is odd, there are $3(N-1)/2$ subgraphs involved. 
\end{lem}
\begin{proof}
The subgraphs are 1-regular, so each vertex corresponding to a point of $X$ appears only once; independence follows as no two distinct edges share a common vertex. 
In the unit norm case, we are applying the same function $v\mapsto v/\norm{v}_2$ to independent vectors, so the results are independent too.
The rest of the lemma is immediate. 
\end{proof}


We can now prove
\begin{thm}\label{thm:SimplexBulkJL}
For $X$ the standard simplex in $\R^D$, it suffices to take 
\[
k=\frac{2C_{\lbrack\ref{thm:FreeDecompBulkJL}\rbrack}}{\epsilon^2}
    \left(\log(2e/\eta) + \frac{\log(D/\delta)}{\max(\eta D, 1)}\right)
\]
for theorem~\ref{thm:UnitFreeDecompBulkJL}. 
\end{thm}
Note $k$ is bounded independent of $D$ as soon as $\eta D\gtrsim \log(D/\delta)$. 
\begin{proof}
Taking $M=D$ in theorem~\ref{thm:UnitFreeDecompBulkJL}, we can use the Walecki construction~\ref{lem:WaleckiConstruction} as is to control $\wh{R}_\infty(D)$. 
When $D$ is odd, the computations from example~\ref{eg:BatchChoiceMattersStableRankSimplex} show $\wh{R}_\infty(D)\geq D/2$, while when $D$ is even, the 1-factor from the Walecki construction is of size $D/2$, with mutually orthogonal vectors of constant norm, so its stable rank is equal to its rank, $D/2$. 
Thus $\wh{R}_\infty(D)\geq D/2$ for both parities of $D$. 
\end{proof}

\section{Further Theorems for i.i.d. Samples}\label{sec:IID}
Let $\xi\in\R^D$ be a given random vector. 
In this section, the following assumption~\ref{AssumptionXZ} will be in play for the $k\times D$ dimension reduction matrix $Z$ and the dataset $X$ of $N$ points in $\R^D$. 
The theorems will then be in terms of additional assumptions on the distribution of $\xi$. 
\begin{assumption}\label{AssumptionXZ}
The matrix $Z$ has $\iid$ zero-mean unit-variance sub-gaussian entries. 
The dataset $X=\set{x_i}_{i=0}^{N-1}\subset\R^D$ is drawn independently of $Z$, with $x_i\overset{\iid}{\drawn} \xi$. 
\end{assumption}
In practice, datasets need not be drawn $\iid$ from some underlying distribution; however, if the number of points $N$ is very large, it may be useful or convenient to subsample the data in order to fit it in memory or to try to estimate properties of the data. 
The theorems in this section may then be read as applying to an $\iid$ (sub)sample of the data, that is, drawing $N$ points uniformly with replacement from the underlying dataset, redefining $N$ to be the new sample size, and redefining $\xi$ to be drawn from the discrete uniform measure on the underlying data.

With the Walecki construction in hand and assumption~\ref{AssumptionXZ}, we can give refinements of theorems~\ref{thm:FreeDecompBulkJL} and~\ref{thm:UnitFreeDecompBulkJL}. 
The theorem statements will have failure probabilities in terms of the draw of the pair $(Z,X)$. 
(There should be no confusion with the dot product notation.) 
Both of the theorems just mentioned give a failure probability $\delta_Z$ for $Z$ once $X$ is fixed, and this probability only depends on stable rank properties of the data set $X$, or more accurately, the difference vector set $Y$. 
These theorems in turn rely on lemma~\ref{lem:BulkHWOrderStat}, that bounds the probabilities in equation~\eqref{eqn:OrderStatEventBatchLevel} for $Z$ acting on a given batch $\Upsilon$ or its unit norm version in terms of stable ranks of minibatches $\Upsilon(\eta)$.  
In many of the examples that follow, the assumptions on the data only guarantee $r_\infty(\Upsilon)$ or $r_\infty(\Upsilon(\eta))$ is above some threshold most of the time, say for a fraction $(1-\zeta)$ of all batches, instead of all of the time. 
We can use lemma~\ref{lem:BulkHWOrderStat} on those \qt{good} batches, and still conclude that $(1-\zeta)(1-2\eta)$ of all distances are approximately preserved with high probability. 
To be concrete, let $\ev_Z$ be the event that $\norm{Z\wh{y}}_2^2$ is approximately preserved for $(1-\zeta)(1-2\eta)$ of the vectors $y\in Y$ --
provided the batches considered have $r_\infty(\Upsilon)\geq r_0$, for some threshold $r_0$. 
Let $\ev_X$ be the event that $r_\infty(\Upsilon)\geq r_0$ for at least $(1-\zeta)$ of the batches considered in the decomposition $\decomp$. 
We then have
\begin{align*}
\Prob_{Z\times X}\set{\ev_Z^c}
&=\Prob_{Z\times X}\set{\ev_Z^c\cap \ev_X}
    +\Prob_{Z\times X}\set{\ev_Z^c\cap \ev_X^c}\\
&=\Prob_{Z\times X}\set{\ev_Z^c\st\ev_X}\Prob_X\set{\ev_X}
    +\Prob_{Z\times X}\set{\ev_Z^c\cap \ev_X^c}\\
&\leq \Prob_{Z\times X}\set{\ev_Z^c\st\ev_X}+\Prob_X\set{\ev_X^c}\\
&\leq \delta_Z+\delta_X
\end{align*}

In the rest of this section, we use the 1-regular subgraph version of the Walecki construction, lemma~\ref{lem:IndependentWithSplitWalecki} 
to define the decomposition $\decomp$ and the underlying batches $\Upsilon$. 
Specifically, each 1-regular subgraph $W$ is at least of size $(N-1)/4$, and we partition the subgraph edges into batches $\Upsilon$ of size $M$ -- each edge corresponding to a difference vector. 
Under assumption~\ref{AssumptionXZ}, the edges within each subgraph are exchangeable, so they can be assigned to batches in an arbitrary manner as long as the batch size is respected.  
For any remainder when $M$ does not divide the subgraph size $\abs{W}$, 
we can modify the definition of $\tilde{\eta}$ from equation~\eqref{eqn:EtaTilde} to apply with $\abs{W}$ in place of $N$; any of the expanded batches are still of size at most $2M-1$. 
Note when $N$ is odd, the subgraphs are not all of the same size, so the $\tilde{\eta}$ will vary accordingly.

We first present one case where $\zeta$ is 0, that is, we are able to control $r_\infty(\Upsilon(\eta))$ for every minibatch, with high probability. 
\begin{thm}\label{thm:SubGaussianIIDCoordsBulkJL}
Under assumption~\ref{AssumptionXZ}, equation~\eqref{eqn:JLUnsquaredDistancesPairwise} holds
for at least $(1-2\eta)\binom{N}{2}$ pairs $x,x'\in X$, with probability at least $1-3\,\delta$ over the draw of $(Z,X)$, provided
\[
k\geq \frac{C_{\ref{thm:FreeDecompBulkJL}}}{\epsilon^2}\frac{4C^2\norm{\xi_1}_{\psi_2}^2(1+(1+\alpha)^2)}{(1-t)}\left(\log(2e/\eta)+\frac{\log(N^2/(D\delta))}{\eta D}\right)
\]
when $\eta D$ is a strictly positive integer, with 
\[
N\geq D\geq \max\left(\frac1{\alpha^2}\log(N^2/(D\delta)),\,
\frac{(2\norm{\xi_1}_{\psi_2}^2+1/\log(2))^2}{ct^2}
    \left(\frac{\log(N^2/(D\delta))}{\eta D}+\log(2e/\eta)\right)
\right). 
\]
\end{thm}
To make some sense of the above, note that if $N\geq D\gtrsim k$, 
it suffices to take $t=\epsilon$ and 
\[
k\geq \frac{C_{\ref{thm:FreeDecompBulkJL}}}{\epsilon^2}
    \frac{4C^2\norm{\xi_1}_{\psi_2}^2(1+(1+\alpha)^2)}{(1-\epsilon)}
    \left(\log(2e/\eta)+\frac{\alpha^2}{\eta}\right)
\]
recalling $D\geq \alpha^{-2}\log(N^2/(D\delta))$ too. 
This target dimension is roughly the same as for the simplex~\ref{thm:SimplexBulkJL} with $D=N$ there, despite the very different sparsity properties of these two datasets. 
\begin{proof}
With $X$ defined as in the theorem statement and $\xi'$ an independent draw of $\xi$, the difference vector set $Y$ is drawn $\iid$ from $\xi-\xi'$, so it is immediately mean-zero with $\iid$ coordinates. 
Because the stable ranks do not see constant scaling, we can work with $Y/\sqrt{2}$, so that $(\xi-\xi')/\sqrt{2}$ now has unit variance coordinates.
The sub-gaussian norm for a coordinate $(\xi_1-\xi_1')/\sqrt{2}$ of $y\in Y/\sqrt{2}$ is at most $\sqrt{2}\norm{\xi_1}_{\psi_2}$ by the triangle inequality. 

We control $R_\infty(\eta M)$ in theorem~\ref{thm:FreeDecompBulkJL} by showing each batch $\Upsilon(\eta)$ has high stable rank. 
Because $r_\infty(\Upsilon(\eta))=\norm{\Upsilon(\eta)}_F^2/\norm{\Upsilon(\eta)}^2$, we shall control numerator and denominator separately. 

We have $\E\norm{\Upsilon(\eta)}_F^2 = D\eta M$. 
Further, Bernstein's inequality for the mean-zero sub-exponential random variables $\Upsilon_{ij}^2-1$ yields~\cite[page~34]{VershyninHDP2018}
\[
\Prob\set{\abs{\frac1{D\eta M}\norm{\Upsilon(\eta)}_F^2-1}\geq t}
\leq 2\exp\left(-c\min\left(\frac{t^2}{K^2}, \frac{t}{K}\right)D\eta M\right)
\]
with $K=\norm{\Upsilon_{11}^2-1}_{\psi_1}$. 
For us with $t\in (0,1)$, the above gives
\[
\Prob\set{(1-t)D\eta M\geq \norm{\Upsilon(\eta)}_F^2}
    \leq 2\exp\left(-\frac{cD \eta M}{\max(K^2,K)}t^2\right). 
\]
To lock down $K$, recall $\norm{a}_{\psi_1}=\abs{a}/\log(2)$ for constants $a$, so 
\begin{align*}
K=\norm{\Upsilon_{11}^2-1}_{\psi_1}\leq \norm{\Upsilon_{11}^2}_{\psi_1}+\norm{\E\Upsilon_{11}^2}_{\psi_1}
\leq \norm{\Upsilon_{11}}^2_{\psi_2}+\frac{\E \Upsilon_{11}^2}{\log(2)}
=2\norm{\xi_1}^2_{\psi_2}+\frac1{\log(2)}. 
\end{align*}
Recall lemma~\ref{lem:StableRanksOfSubsets} which allowed us to state $r_\infty(\wh{\Upsilon}(\eta))\geq \eta r_\infty(\wh{\Upsilon})$ always, because the columns of $\wh{\Upsilon}$ had constant norm. 
We can give a high probability replacement for that lemma in this context, but require that it must hold over all $\floor{\binom{N}{2}/M}<N^2/(2M)$ batches $\Upsilon$, that is, 
\begin{align*}
\frac{2M\delta}{N^2}&\geq 
    \binom{M}{\eta M}2\exp\left(-\frac{cD \eta M}{\max(K^2,K)}t^2\right)\\
\log(M\delta/N^2)&\geq \log\binom{M}{\eta M}-\frac{cD \eta M}{\max(K^2,K)}t^2\\
\frac{cD \eta M}{\max(K^2,K)}t^2&\geq \log(N^2/(M\delta))+\log\binom{M}{\eta M}. 
\end{align*}
Taking into account the $\tilde{\eta}$ cases, as mentioned before this proof, we require
\[
D\geq \frac{(2\norm{\xi_1}_{\psi_2}^2+1/\log(2))^2}{ct^2}
    \left(\frac{\log(N^2/(M\delta))}{\eta M}+\log(2e/\eta)\right). 
\]
For ambient dimensions $D$ at least this large, 
every single batch $\Upsilon$ satisfies
\[
r_\infty(\Upsilon(\eta))\geq \frac{(1-t)D\eta M}{\norm{\Upsilon(\eta)}^2}
\geq (1-t)\eta \frac{DM}{\norm{\Upsilon}^2}
\]
with total failure probability at most $\delta$. 

We now only need to control $\norm{\Upsilon}^2$ instead of $\norm{\Upsilon(\eta)}^2$, and we can use the known result~\cite[page~85]{VershyninHDP2018} for matrices with mean-zero independent sub-gaussian entries -- recalling the $\psi_2$-norm of each entry is now $2\norm{\xi_1}_{\psi_2}$, as mentioned above -- 
\[
\norm{\Upsilon}\leq C(2\norm{\xi_1}_{\psi_2})(\sqrt{D}+\sqrt{M}+s)
\]
with probability at least $1-2\exp(-s^2)$. 
It is proved via an $\epsilon$-net argument. 
For us, with $s=\alpha\sqrt{D}$, the above gives
\[
\norm{\Upsilon}^2>4C^2\norm{\xi_1}_{\psi_2}^2(M+(1+\alpha)^2D) 
\]
with probability at most $2\exp(-\alpha^2 D)$. 
Consequently, 
\begin{equation}\label{eqn:RinftyAtetaSubGaussian}
r_\infty(\Upsilon(\eta))\geq \frac{(1-t)\eta DM}{4C^2\norm{\xi_1}_{\psi_2}^2(M+(1+\alpha)^2D)}
=\frac{(1-t)\eta M}{4C^2\norm{\xi_1}_{\psi_2}^2((M/D)+(1+\alpha)^2)}
\end{equation}
over all minibatches $\Upsilon(\eta)$ from all batches $\Upsilon$ 
with total failure probability at most 
\[
\frac{N^2}{2M}2\exp(-\alpha^2 D)+\delta
\leq 2\delta
\qtext{provided} D\geq \frac1{\alpha^2}\log(N^2/(M\delta)) \qtext{as well.} 
\]
So we may take the right hand side of equation~\eqref{eqn:RinftyAtetaSubGaussian} as our $R_\infty(\eta M)$ value. 
Plugging into theorem~\ref{thm:FreeDecompBulkJL} yields
\[
k\geq \frac{C_{\ref{thm:FreeDecompBulkJL}}}{\epsilon^2}\frac{4C^2\norm{\xi_1}_{\psi_2}^2((M/D)+(1+\alpha)^2)}{(1-t)}\left(\log(2e/\eta)+\frac{\log(N^2/(M\delta))}{\eta M}\right).
\]
Setting $M=D$ completes the proof. 
\end{proof}


\subsection{Controlling \qt{Most} Batches}
Unlike in theorem~\ref{thm:SubGaussianIIDCoordsBulkJL} above, in general the dataset $X$ need not be so well-behaved, to the point that controlling $r_\infty(\Upsilon(\eta))$ is not possible across all minibatches. 
To make things more manageable, we shall now work with theorem~\ref{thm:UnitFreeDecompBulkJL}, using the unit normalized batches $\wh{\Upsilon}$, working on batches of size at least $M$ instead of $\eta M$. 
For general random vectors $\xi$, we shall not be able to guarantee that $\wh{\Upsilon}$ has high stable rank, but we shall guarantee that a fraction $(1-\zeta)$ of the batches have stable rank comparable 
to a \qt{typical} value associated with $\xi$. 

The columns of $\wh{\Upsilon}$ have constant unit norm, so it is enough to control $\norm{\wh{\Upsilon}}^2=\norm{\wh{\Upsilon}\wh{\Upsilon}^\top}$ in order to control $r_\infty(\wh{\Upsilon})$. 
It will be convenient to introduce some new notation for this purpose; we present it first in its unnormalized version. 

Because we are still using the Walecki construction via lemma~\ref{lem:IndependentWithSplitWalecki} to define the batches, 
the columns in $\Upsilon$ are independent, each drawn like the given random vector $y:=\xi-\xi'\in \R^D$, with $\xi'$ an independent copy of $\xi$. 
The corresponding second moment matrix $\Sigma:=\Sigma(y):=\E yy^\top$ 
is twice the covariance matrix for $\xi$, for if $\E \xi=\mu$, 
\begin{align*}
\E yy^\top &= \E(\xi-\xi')(\xi-\xi')^\top = 2(\Sigma(\xi)-\mu\mu^\top)
=2(\E(\xi-\mu)(\xi-\mu)^\top)
=2\,\Cov(\xi). 
\end{align*}
Consequently, $r(y)$ is the effective rank of $\Cov(\xi)$, as the factors of 2 will cancel in the ratio. 

Now $\Sigma(y)$ may be approximated by its empirical version
\[
\Sigma_M:=\frac1{M}\Upsilon\Upsilon^\top = \frac1{M}\sum_{i=1}^M y_iy_i^\top. 
\qtext{Recall also}
\E\norm{y}_2^2 = \E\tr{y^\top y}=\E\tr{yy^\top}=\tr\Sigma. 
\]
The unit normalized versions will depend on $\wh{y}$, with $\wh{\Sigma}:=\E\wh{y}\wh{y}^\top$, $M\wh{\Sigma}_M:=\wh{\Upsilon}\wh{\Upsilon}^\top$, and 
$\wh{r}:=r(\wh{y})=1/\norm{\wh{\Sigma}}$. 
The connection between $r$ and $\wh{r}$ is not so immediate, but we shall return to this point in section~\ref{sec:Retraction}. 

The following is implicit in~\cite[section~5.6]{VershyninHDP2018}; 
we include the proof here, as we want explicit constants, and we plan to take $K=1$, applying it to $\wh{\Upsilon}$.
Controlling the deviation $\norm{\wh{\Sigma}_M-\wh{\Sigma}}$ will be the way we eventually show $r_\infty(\wh{\Upsilon})\gtrsim\wh{r}$ \qt{most} of the time. 
\begin{lem}\label{lem:RelativeErrorCovarSigmaOneBoundedCase}
Let $\Upsilon=\set{y_i}_{i=1}^M\subset\R^D$ be as above, with the $y_i\overset{\iid}{\drawn} y$, and assume for some $K\geq 1$, that 
\[
\norm{y}_2^2\leq K^2\E\norm{y}_2^2 \qtext{almost surely.}
\]
Then, 
with failure probability at most $2e^{-u}$, 
\[
\frac{\norm{\Sigma_M-\Sigma}}{\norm{\Sigma}}
\leq \left(\frac{(4K^2/3) r(u+\log D)}{M}
    +\sqrt{\frac{2K^2 r(u+\log D)}{M}}\right)
\qtext{for $r=r(y)=\frac{\tr\Sigma(y)}{\norm{\Sigma(y)}}$.} 
\]
\end{lem}
\begin{proof}
Because the matrices $y_iy_i^\top -\Sigma$ are symmetric, $\iid$, and mean-zero, we can use the matrix Bernstein inequality: for $t\geq 0$, 
\[
\Prob\set{\norm{\sum_{i=1}^M(y_iy_i^\top -\Sigma)}\geq t}
\leq 2D\exp\left(-\frac{t^2/2}{\sigma^2+bt/3}\right)
\]
with $b\geq \norm{y_iy_i^\top -\Sigma}$ and $\sigma^2=\norm{\sum_{i=1}^M \E(y_iy_i^\top -\Sigma)^2}$. 
We want the right hand side to be at most $2e^{-u}$, that is, 
\[
s=\log(D)+u\leq \frac{t^2/2}{\sigma^2+bt/3}, \qtext{or}
0\leq t^2-\frac{2sb}{3}t-2\sigma^2s. 
\]
The positive root is at
\[
t^\ast=\frac1{2}\left(\frac{2sb}{3}+\sqrt{\Big(\frac{2sb}{3}\Big)^2+8\sigma^2 s}\right). 
\]
Because the other root of the quadratic is negative, taking any $t\geq t^\ast$ will suffice for us. 
Because the square root function is subadditive, it is safe to take $t=(2sb/3)+\sigma\sqrt{2s}$. 

We now just need to estimate $b$ and $\sigma^2$. 
Estimate $b$ via
\begin{align*}
\norm{y_iy_i^\top -\Sigma}&\leq \norm{y_iy_i^\top}+\norm{\Sigma}
=\norm{y_i}_2^2+\norm{\Sigma}\\
&\leq K^2\tr\Sigma+\norm{\Sigma}
\leq 2K^2\tr\Sigma=:b. 
\end{align*}
For $\sigma^2$, the $\iid$ assumption already gives 
$\sigma^2=\norm{M\E(yy^\top-\Sigma)^2}$. 
Expanding the square, 
\[
\E(yy^\top -\Sigma)^2=\E(yy^\top)^2-\Sigma^2, 
\]
while $(yy^\top)^2=y(y^\top y)y=\norm{y}_2^2 yy^\top$. 
Taking expectations on $v^\top (yy^\top)^2 v$ with $v\in S^{D-1}$ gives
\[
v^\top \E(yy^\top)^2 v=\E\norm{y}_2^2 v^\top yy^\top v
\leq K^2(\tr\Sigma)\E v^\top yy^\top v=K^2(\tr \Sigma)v^\top \Sigma v. 
\]
Taking the maximum over $v\in S^{D-1}$ gives
\[
\norm{\E(yy^\top)^2}\leq K^2\tr\Sigma\norm{\Sigma} 
\qtext{and}
\sigma^2\leq MK^2\tr\Sigma\norm{\Sigma}, 
\]
as $-\Sigma^2$ is negative semidefinite. 

Recalling our choice of $t$ and that $r=\tr\Sigma/\norm{\Sigma}$, we find
\begin{align*}
t&=\frac{2sb}{3}+\sigma\sqrt{2s}\\
&=\frac{4}{3}K^2 \tr\Sigma (u+\log D)+\sqrt{2(u+\log D)M K^2 \tr\Sigma\norm{\Sigma}}\\
&=M\norm{\Sigma}\left(\frac{(4K^2/3)r(u+\log D)}{M}+\sqrt{\frac{2K^2 r(u+\log D)}{M}}\right). 
\end{align*}
With this $t$ value, we finally have
\[
\Prob\set{\norm{\Sigma_m-\Sigma}\geq \frac{t}{M}}\leq 2e^{-u}, \qtext{as desired.} \qedhere
\]
\end{proof}

If we used the unit normalized $\wh{\Upsilon}$, lemma~\ref{lem:RelativeErrorCovarSigmaOneBoundedCase} provides the following upper tail probability 
\begin{equation}\label{eqn:UpperTailCovarSigmaOne}
P_X\set{\norm{\wh{\Sigma}_M-\wh{\Sigma}}>\tau_+(u)}\leq 2\exp(-u)
\end{equation}
with $\tau_+(u)$ a function of $u$. 
By the triangle inequality, we should immediately have
\[
\norm{\wh{\Upsilon}}^2=\norm{M\wh{\Sigma}_M}\leq M\norm{\wh{\Sigma}_M-\wh{\Sigma}}+M\norm{\wh{\Sigma}}
\leq M\norm{\wh{\Sigma}}(1+\tau_+(u)), 
\]
so that $r_\infty(\wh{\Upsilon})\geq \wh{r}/(1+\tau_+(u))$
with failure probablity $2\exp(-u)$. 
This tail probability is too weak to control $r_\infty(\wh{\Upsilon})$ for every batch with high probability, because $u$ will need to be too large to be useful, so we allow a fraction $\zeta>0$ of the batches to fail. 
We consider order statistics of real-valued functions of $\Upsilon$ across batches, namely $\norm{\wh{\Sigma}_M-\wh{\Sigma}}$ -- 
because the batches within each 1-regular subgraph are independent, we can 
exploit that independence to inform the choice for $u$ using the following lemma.   
\begin{lem}\label{lem:OrderStatForBatches}
Let $\set{\omega_i}_0^{n-1}$ be $\iid$ random variables. 
Let $\zeta\in (0,1)$ with $\zeta n$ an integer. 
If 
\[
\Prob\set{\omega_0>\tau_+(u)}\leq 2e^{-u}, \qtext{then} 
\Prob\set{\omega_{((1-\zeta)n)}>\tau_+(u)}
    \leq \exp\big(\zeta n(\log(e/\zeta)+\log(2)-u)\big). 
\]
\end{lem}
\begin{proof}
First note
\[
\Prob\set{\omega_{((1-\zeta)n)}>\tau_+(u)}=\Prob\set{\omega_{(i)}>\tau_+(u)\txt{for} i\geq (1-\zeta)n}, 
\]
that is, we are looking for the top $\zeta n$ of the $\omega_{(i)}$ to be larger than $\tau_+(u)$. 
Because the $\omega_i$ are drawn $\iid$, all their $\zeta n$ sized subsets are equally likely, so we may conclude
\begin{align*}
\Prob\set{\omega_{((1-\zeta)n)}>\tau_+(u)}
&\leq \binom{n}{\zeta n}
    \Prob\set{\omega_i>\tau_+(u)\txt{for} 0\leq i\leq \zeta n-1}\\
&\leq \left(\frac{e}{\zeta}\right)^{\zeta n}\left(2e^{-u}\right)^{\zeta n}
=\exp\big(\zeta n(\log(e/\zeta)+\log(2)-u)\big), 
\end{align*}
using the independence of the $\omega_i$ in the last line. 
%
\end{proof}

In lemma~\ref{lem:OrderStatForBatches}, we shall take $n$ to be the number of batches in a given subgraph, so that the random variables in question are independent. 
The subgraphs from lemma~\ref{lem:IndependentWithSplitWalecki}
are of size at least $(N-1)/4$ and at most $N/2$, but are not all of the same size when $N$ is odd. 
If $W$ is one of these subgraphs, it contains $\floor{\abs{W}/M}$ batches, as we have enforced each batch to have size at least $M$, and reassigned any remainder $\abs{W}\mod M$ to those batches. 
%
When $N$ is even, $\abs{W}=N/2$, and we only need to require $\zeta n:=\zeta\floor{N/(2M)}\in \N$. 
When $N$ is odd, we adjust $\zeta$ depending on the size of the subgraph $W$.  
Suppose we enforce $\zeta n:=\zeta \floor{(N-1)/(4M)}$ to be an integer, recalling $\abs{W}\geq (N-1)/4$ in all cases. 
For any other subgraph $W$ of larger size, set
\begin{equation}\label{eqn:ZetaTilde}
\tilde{\zeta} = \zeta \frac{n}{\floor{\abs{W}/M}}
\qtext{so that}
\tilde{\zeta}\floor{\abs{W}/M}=\zeta n. 
\end{equation}
In particular for all the subgraph sizes, $1/\tilde{\zeta}< 2(n+1)/(\zeta n)\leq 3/\zeta$ if we assume $\zeta\leq 1/2$, as $\zeta n\geq 1$. 
Thus in lemma~\ref{lem:OrderStatForBatches}, we replace $\log(e/\zeta)$ by $\log(3e/\zeta)$ when we consider the different subgraphs. 

\subsubsection{Retraction to the Sphere}\label{sec:Retraction}
Suppose $y$ does not have a second moment, that is, $\Cov(\xi)$ is undefined. As mentioned before, we can use lemma~\ref{lem:RelativeErrorCovarSigmaOneBoundedCase} on the unit normalized batches $\wh{\Upsilon}$ instead, provided we replace $\Sigma$ by $\wh{\Sigma}:=\E \wh{y}\wh{y}^\top$ and $r$ by $\wh{r}=r(\wh{y})=1/\norm{\wh{\Sigma}}$. 
By Cauchy-Schwarz, $\norm{\wh{\Sigma}}\leq 1$ always: for any unit vector $v$,
\[
v^\top \E \wh{y}\wh{y}^\top v 
=\E (\wh{y},v)^2\leq \E\norm{\wh{y}}_2^2\norm{v}_2^2=1. 
\]
A first question to ask is: in the presence of a second moment, how are the operator norms of $\Sigma$ and $\wh{\Sigma}$ related? 
The following lemma gives one such answer. 
\begin{lem}\label{lem:StableRankTransferForRetraction}
Let $y$ be a random vector in $\R^D$, with second moment matrix $\E yy^\top =\Sigma$. If $\wh{y}=y/\norm{y}_2$, then with $\wh{\Sigma}=\E \wh{y}\wh{y}^\top$, for any $\epsilon\in(0,1)$, 
\[
\norm{\wh{\Sigma}}
\leq \frac1{\epsilon\, r}+p(\epsilon)
\qtext{with}
p(\epsilon):=\Prob\set{\norm{y}_2^2<\epsilon\, \E\norm{y}_2^2} \qtext{and}
r=\frac{\tr\Sigma}{\norm{\Sigma}}=\frac{\E\norm{y}_2^2}{\norm{\Sigma}}. 
\]
Further, $\wh{r}\geq r/(\epsilon^{-1}+rp(\epsilon))$. 
\end{lem}
\begin{rem}
Some dependence on the small ball (if $\epsilon\ll 1$) or lower deviation (if $\epsilon\approx 1$) probability 
$\Prob\set{\norm{y}_2^2<\epsilon\, \E\norm{y}_2^2}$ must be present, in general, due to the nature of the retraction to the sphere. 
Specifically, suppose $y$ is distributed uniformly at random from a finite set in $\R^D$, having a cluster of points all pointing in roughly the $e_1$ direction, but of very small norm. 
If all the other points are distributed uniformly on the unit sphere, one expects $\norm{\Sigma}$ to be roughly 1/D, as the cluster points will not contribute much to the operator norm. 
However, after the retraction, these cluster points now all have unit norm and are still pointing in roughly the $e_1$ direction, so if there are enough points in the cluster, $\norm{\wh{\Sigma}}$ could now be much closer to 1. 
\end{rem}
\begin{proof}
The matrices $\Sigma$ and $\wh{\Sigma}$ are symmetric positive semi-definite, 
so their singular values correspond to their eigenvalues. 
These eigenvalues involve the quadratic form $v\mapsto v^\top \Sigma v$ or $v\mapsto v^\top \wh\Sigma v$. 
The latter quadratic form is just $\E (\wh{y}, v)^2$, which we may split as
\[
\E (\wh{y}, v)^2 = \E (\wh{y}, v)^2\I\set{\norm{y}_2^2\geq \epsilon\,\E \norm{y}_2^2} 
+ \E (\wh{y}, v)^2\I\set{\norm{y}_2^2< \epsilon\,\E \norm{y}_2^2}.
\]
For the first term, 
\[
\E (\wh{y}, v)^2\I\set{\norm{y}_2^2\geq \epsilon\,\E \norm{y}_2^2}
\leq \frac{\E (y, v)^2}{\epsilon\E\norm{y}_2^2}\I\set{\norm{y}_2^2\geq \epsilon\,\E \norm{y}_2^2}
\leq \frac{\E (y, v)^2}{\epsilon\,\E\norm{y}_2^2}. 
\]
When $v$ is a unit vector, $(\wh{y},v)^2\leq 1$ always, so for such $v$, 
\begin{align*}
\E (\wh{y}, v)^2&\leq \frac{\E (y, v)^2}{\epsilon\,\E\norm{y}_2^2}
    +\E (\wh{y}, v)^2\I\set{\norm{y}_2^2< \epsilon\,\E \norm{y}_2^2}\\
&\leq \frac{\E (y, v)^2}{\epsilon\,\E\norm{y}_2^2}
    +\Prob\set{\norm{y}_2^2< \epsilon\,\E \norm{y}_2^2}. 
\end{align*}
Now $\norm{\wh{\Sigma}}$ is just the maximum of $\E (\wh{y},v)^2$ over the unit sphere, and the maximum is realized by some $v^\ast$ as the sphere is compact. 
Consequently, 
\begin{align*}
\norm{\wh{\Sigma}}
&\leq \frac{\E (y, v^\ast)^2}{\epsilon\,\E\norm{y}_2^2}
    +\Prob\set{\norm{y}_2^2< \epsilon\,\E \norm{y}_2^2}
\leq \max_{v\in S^{D-1}}\frac{\E (y, v)^2}{\epsilon\,\E\norm{y}_2^2}
    +\Prob\set{\norm{y}_2^2< \epsilon\,\E \norm{y}_2^2}\\
&=\frac{\norm{\Sigma}}{\epsilon\,\E\norm{y}_2^2}
    +\Prob\set{\norm{y}_2^2< \epsilon\,\E \norm{y}_2^2}. 
\end{align*}
Recalling $\E\norm{y}_2^2=\tr \Sigma$ and the definition of $r$ and $\wh{r}=1/\norm{\wh{\Sigma}}$ finishes the proof. 
\end{proof}


Apart from assumption~\ref{AssumptionXZ}, we make no other assumptions on $X$ or $\xi$ in the following theorem. 
\begin{thm}\label{thm:ArbitraryOnUnitSphereBulkJL}
Let $0<\epsilon,\,\eta,\,\zeta<1$, $\zeta\leq 1/2$, $\epsilon\leq 2/3$, and $0<\alpha$. 
Under assumption~\ref{AssumptionXZ}, equation~\eqref{eqn:JLUnsquaredDistancesPairwise} holds
for at least $(1-2\eta)(1-\zeta)\binom{N}{2}$ pairs $x,x'\in X$, with probability at least $1-2\delta$ over the draw of $(Z,X)$, provided
\[
k\geq \frac{C_{\lbrack\ref{thm:FreeDecompBulkJL}\rbrack}}{\epsilon^2}
\left(\alpha+\frac4{3}+\sqrt{2\alpha}\right)\left(\log(2e/\eta) \frac{(\log(6e/\zeta)+\log D)}{1-t'}  
    +\frac{\log(N^2/(\alpha\wh{r}\delta\log(D)))}{\alpha\max(\eta\wh{r},1)}\right)
\]
when the quantities $\eta M$ and $\zeta\floor{(N-1)/(4M)}$ are strictly positive integers, with
\[
\frac{N-1}{8}\geq M:=
\alpha\wh{r}\frac{(\log(6e/\zeta)+\log D)}{1-t'} 
\quad
1>t':=\frac{8\alpha\wh{r}\log(\frac{3(N-1)}{2\delta})}{\zeta(N-1)}, 
\qtext{and} \gamma(\epsilon)=\gamma_{\lbrack\ref{thm:FreeDecompBulkJL}\rbrack}. 
\]
\end{thm}
\begin{rem}
To make sense of the above, consider $D=D_{JL}=O(\epsilon^{-2}\log(N^2/\delta))$ as if from the original Johnson-Lindenstrauss lemmas. 
We always have $\wh{r}\leq D$, so $t'$ vanishes fairly quickly with increasing $N$ when $\zeta$ is fixed or even slowly decaying. 
Compared to theorem~\ref{thm:SubGaussianIIDCoordsBulkJL}, we have an extra factor against $\log(2e/\eta)$, but it is not too big in that $\log(D)=O(\log\log(N/\delta))$, and $\log(6e/\zeta)$ does not grow quickly with decreasing $\zeta$.  
When $N$ is large, we can make $\zeta$ small enough that the fraction $(1-2\eta)(1-\zeta)$ is not that much worse than $(1-2\eta)$. 
Finally, in light of lemma~\ref{lem:StableRankTransferForRetraction}, we can replace $\wh{r}$ by $r$ in the definition of $k$ at the expense of a bounded prefactor for $\log(N^2/\delta)$ provided the lower deviation or small ball probability $p(\epsilon)$ is less than $1/r$. 
\end{rem}
\begin{proof}
From lemma~\ref{lem:IndependentWithSplitWalecki}, we have either $N-1$ or $3(N-1)/2$ 1-regular subgraphs to consider when $N\geq 7$, and we choose (unit normalized) batches $\wh{\Upsilon}$ of size at least $M$ from these subgraphs. 
Let $M\wh{\Sigma}_M:=\wh{\Upsilon}\wh{\Upsilon}^\top$. 
Within each subgraph, the batches are independent, allowing us to use lemma~\ref{lem:OrderStatForBatches} on the random variables 
\[
\omega(\wh{\Upsilon}):=\norm{\wh{\Sigma}_M-\wh{\Sigma}}. 
\]
Take $n:=\floor{(N-1)/(4M)}$ for that lemma and recall the $\tilde{\zeta}$ discussion from equation~\eqref{eqn:ZetaTilde}. 
We have to choose $u$ so that a union bound over all the subgraphs is still smaller than $\delta$, so a safe value for $u$ would be 
\begin{align}
\frac{3(N-1)}{2}\exp\left(\zeta n\big(\log(3\,e/\zeta)+\log(2)-u\big)\right)
&\leq \delta\\
\label{eqn:uDepsZetaAndn}
\log(6e/\zeta)+\frac{\log(3(N-1)/(2\delta))}{\zeta n}&\leq u
\end{align}
With this $u$ in hand, 
we can apply lemma~\ref{lem:RelativeErrorCovarSigmaOneBoundedCase} with $K=1$, for $M\wh{\Sigma}_M=\wh{\Upsilon}\wh{\Upsilon}^\top$
\begin{align*}
\norm{\wh{\Upsilon}}^2=\norm{M\wh{\Sigma}_M}&\leq M\norm{\wh{\Sigma}_M-\wh{\Sigma}}+M\norm{\wh{\Sigma}}\\
&\leq M\norm{\wh{\Sigma}}\left(1+\left(\frac{(4/3) \wh{r}(u+\log D)}{M}
    +\sqrt{\frac{2\wh{r}(u+\log D)}{M}}\right) \right). 
\end{align*}
Because we are interested in 
$r_\infty(\wh{\Upsilon})=M/\norm{\wh{\Upsilon}}^2$, we should like to make the error term manageable, so choose 
\[
M=\alpha \wh{r}(u+\log D) \qtext{with} \alpha>0. 
\]
Because $u$ already depends on $M$ through $n$ in equation~\ref{eqn:uDepsZetaAndn}, there is a constraint on $u$ and $\zeta$ that we need to address.  
Write $(N-1)/4 = nM +s$ with $0\leq s\leq M-1$. 
Set $\zeta^\ast$ to satisfy $\zeta n = \zeta^\ast (N-1)/(4M)$. 
We then have
\begin{align*}
\log(6e/\zeta)+t(u+\log D)&\leq u
\qtext{with} t:=4\alpha\wh{r}\frac{\log(3(N-1)/(2\delta))}{\zeta^\ast (N-1)}>0\\
\log(6e/\zeta)+t\log D
&\leq (1-t)u
\end{align*}
We can divide by $(1-t)$ provided $t<1$. 
\begin{align*}
\alpha\wh{r}\log(3(N-1)/(2\delta))<\zeta^\ast\frac{N-1}{4}
=\zeta\left(\frac{N-1}{4}-s\right)
\end{align*}
Recalling $s\leq M-1$, if we also require $M\leq (N-1)/8$, 
it would be safe to require
\[
\alpha\wh{r}\log(3(N-1)/(2\delta))< \zeta\frac{N-1}{8}
<\zeta\left(\frac{N-1}{4}-s\right)
\qtext{and}
u\leq \frac{N-1}{8\alpha\wh{r}}-\log(D). 
\]
We then have
\[
1>t':=8\alpha\wh{r}\frac{\log(3(N-1)/(2\delta))}{\zeta(N-1)}>t, 
\]
and because the maps $t\mapsto t/(1-t)$ and $t\mapsto 1/(1-t)$ are strictly increasing, a valid lower bound for $u$ is
\[
u\geq \frac{t'}{1-t'}\log(D)+\frac1{1-t'}\log(6e/\zeta). 
\]
Taking $u$ as this lower bound yields
\[
u+\log(D)=\frac1{1-t'}(\log(D) + \log(6e/\zeta)). 
\]

With this choice of $u$ in hand, with probability at least $1-\delta$, 
\[
r_\infty(\wh{\Upsilon})\geq \frac{M}{M\norm{\wh{\Sigma}}
    \left(1+\frac4{3\alpha}+\sqrt{\frac2{\alpha}}\right)}
=\frac{1/\norm{\wh{\Sigma}}}{\left(1+\frac4{3\alpha}+\sqrt{\frac2{\alpha}}\right)}
=\frac{\wh{r}}{\left(1+\frac4{3\alpha}+\sqrt{\frac2{\alpha}}\right)}
=:\wh{R}_\infty(M;\zeta)
\]
for at least $(1-\zeta)$ of all batches $\wh{\Upsilon}$. 
Assuming this bound holds, we now ask that when $Z$ is drawn, equation~\eqref{eqn:JLUnsquaredDistancesPairwise} holds for all the vectors involved in \emph{at least} these batches, with failure probability at most $\delta$. 
We run the argument of theorem~\ref{thm:UnitFreeDecompBulkJL}, only for these batches $\wh{\Upsilon}$, using $\wh{R}_\infty(M;\zeta)$ in place of $\wh{R}_\infty(M)$. 
As we could still have $\floor{\binom{N}{2}/M}$ \qt{good} batches, we still must allow for all of them when we compute the union bound.  
The $M/\wh{R}_\infty(M)$ ratio in theorem~\ref{thm:UnitFreeDecompBulkJL} now just becomes
\[
\frac{M}{\wh{R}_\infty(M; \zeta)}
=\frac{\alpha \wh{r}(u+\log D)}{\wh{r}}
    \left(1+\frac4{3\alpha}+\sqrt{\frac2{\alpha}}\right)
=(u+\log D)\left(\alpha+\frac4{3}+\sqrt{2\alpha}\right)
\]
Let $C_\alpha$ be the coefficient of $(u+\log D)$ in the above.
We may then set $k$ as 
\begin{equation}\label{eqn:tentativeK}
k\geq \frac{C_{\lbrack\ref{thm:FreeDecompBulkJL}\rbrack}}{\epsilon^2}
C_\alpha \left(\log(2e/\eta) \frac{(\log(6e/\zeta)+\log D)}{1-t'}  
    +\frac{\log(N^2/(\alpha\wh{r}\delta\log(D)))}{\alpha\max(\eta\wh{r},1)}\right)
\end{equation}
using $u+\log(D)\geq \log(D)$ in the $\log(N^2)$ term. 
The choice of $\gamma$ follows from theorem~\ref{thm:UnitFreeDecompBulkJL}. 
\end{proof}

In certain cases, we know $\wh{r}$ exactly without relying on lemma~\ref{lem:StableRankTransferForRetraction}. 
\begin{lem}\label{lem:IsotropicRetractionIID}
Suppose $\xi=(\xi_1,\ldots, \xi_D)\in \R^D$ is a random vector with $\iid$ coordinates, and $\xi'$ is an independent copy of $\xi$. 
If $y:=\xi-\xi'$, then the scaled unit vector $\wh{y}\sqrt{D}$ 
is mean-zero isotropic.  
\end{lem}
There are no moment assumptions on the coordinates $\xi_i$ here, so the lemma even applies to vectors with $\iid$ standard Cauchy coordinates.   
\begin{proof}
Both properties rely on the following observation. 
For a fixed coordinate $i$, the coordinate $ y_i=\xi_i-\xi_i'$ is a symmetric random variable: $ y_i$ and $- y_i$ have the same distribution. 
Consequently, for any odd function $f$, (using the symmetry in the 2nd equality)
\[
-\E f( y_i)=\E f(- y_i)=\E f( y_i)
\]
so that $\E f( y_i)=0$ for such odd functions $f$. 

If we temporarily freeze the other coordinates, the $i$th coordinate of the unit vector $\wh{ y}$ is just
\[
\wh{ y}_i=\frac{ y_i}{( y_i^2 + \sum_{j\neq i}^D y_j^2)^{1/2}}
=\frac{ y_i}{( y_i^2+C)^{1/2}},
\]
an odd function of $ y_i$, forcing $\wh{ y}\sqrt{D}$ to be mean-zero. 

To check $\wh{ y}\sqrt{D}$ is isotropic, we must show the matrix 
$\Sigma=D\,\E \wh{ y}\wh{ y}^\top$ is the identity $\Id_D$. 
Because the $ y_i$ are identically distributed, 
\[
D\,\E\frac{ y_1^2}{\sum_{j=1}^D  y_j^2}
=\E \sum_{i=1}^D \frac{ y_i^2}{\sum_{j=1}^D  y_j^2}=1, 
\]
so the diagonal entries of $\Sigma$ are all equal to $1$. 

Further, when $i\neq j$, the independence of the $ y_i$ now give
\[
D\,\E \frac{ y_i y_j}{ y_i^2+\sum_{k\neq i}^D y_k^2}
=D\,\E y_j\E\left(\frac{ y_i}{ y_i^2+\sum_{k\neq i}^D y_k^2}\Big\lvert y_{j\neq i}\right)
=0
\]
because the conditional expectation vanishes on the odd function of $ y_i$. 
\end{proof}


We now have an immediate corollary to theorem~\ref{thm:ArbitraryOnUnitSphereBulkJL}, which again we may compare to theorem~\ref{thm:SubGaussianIIDCoordsBulkJL}. 
\begin{cor}\label{cor:ArbitraryIIDCoordsBulkJL}
In the setting of theorem~\ref{thm:ArbitraryOnUnitSphereBulkJL}, suppose $\xi$ has $\iid$ coordinates. 
Then the corresponding conclusion still holds, with $\wh{r}=D$, namely it suffices to take $\gamma(\epsilon)=\gamma_{\lbrack\ref{thm:FreeDecompBulkJL}\rbrack}$ and
\[
k\geq \frac{C_{\lbrack\ref{thm:FreeDecompBulkJL}\rbrack}}{\epsilon^2}
\left(\alpha+\frac4{3}+\sqrt{2\alpha}\right)\left(\log(2e/\eta) \frac{(\log(6e/\zeta)+\log D)}{1-t'}  
    +\frac{\log(N^2/(\alpha D\delta\log(D)))}{\alpha\max(\eta D,1)}\right). 
\]
\end{cor}
\begin{proof}
By lemma~\ref{lem:IsotropicRetractionIID}, the difference vector $y=\xi-\xi'$ yields the isotropic vector $\wh{y}\sqrt{D}$.
Because $\Sigma(\wh{y}\sqrt{D})=\Id_D$, we compute 
$\wh{r}:=r(\wh{y})=r(\wh{y}\sqrt{D})=D$, as the intrinsic dimension does not see constant scalings. 
We can then apply theorem~\ref{thm:ArbitraryOnUnitSphereBulkJL}. 
\end{proof}
\begin{rem}
The proof only uses that $\wh{y}\sqrt{D}$ is isotropic. 
By lemma~\ref{lem:UnitaryPreservesIsotropic}, $\Phi(\wh{y}\sqrt{D})$ is still isotropic when $\Phi$ has orthonormal rows. 
Because equation~\ref{eqn:JLUnsquaredDistancesPairwise} is 1-homogeneous, the corollary still holds with $\xi$ replaced by $\Phi(\xi)$, in particular when $\Phi$ has a fast transform method available. 
\end{rem}

\subsection{Estimating \texorpdfstring{$\wh{r}$}{the intrinsic dimension of y hat}}
The intrinsic dimension of $\wh{y}$, namely $\wh{r}$, enters into theorem~\ref{thm:ArbitraryOnUnitSphereBulkJL} only as a parameter, so we are free to estimate it separately. 
In particular, 
\begin{cor}\label{cor:EstimatingRHat}
Theorem~\ref{thm:ArbitraryOnUnitSphereBulkJL} holds with $\wh{r}$ replaced by 
an empirical estimate using a batch $\wh{\Upsilon}(m)$ of size $m$, namely,
with failure probability at most $\delta$, 
\[
\wh{r}\geq \frac1{3\norm{\wh{\Sigma}_m}}\geq \frac{\wh{r}}{5}
\qtext{for} m=8D\log(2D/\delta)
\qtext{provided} m\leq (N-1)/2. 
\]
\end{cor}
If $D=D_{JL}=O(\epsilon^{-2}\log(N^2/\delta))$, then computing $\norm{\wh{\Sigma}_m}$ will cost polynomial in $\log(N^2/\delta)$ and $\epsilon^{-2}$; however, because $1\leq \wh{r} \leq D$, we do not need very high accuracy when computing this top eigenvalue. 

If we were working with $X$ drawn uniformly with replacement from a larger dataset, we should draw the first $16D\log(2D/\delta)$ data points, and then sequentially pair them off for unit difference vectors to yield $\wh{\Upsilon}(m)$. 
The uniform with replacement assumption assures that these data points used are as good as any other subset, even if some of the data points turn out to be copies of the same point in the larger dataset. 

\begin{proof}
For $m\leq (N-1)/2$, we can find a batch $\wh{\Upsilon}(m)$ of size $m$ with indepedent columns, by corollary~\ref{cor:WaleckiSplitCycles} and lemma~\ref{lem:IndependentWithSplitWalecki}. 
By lemma~\ref{lem:RelativeErrorCovarSigmaOneBoundedCase}, we have, 
with failure probability at most $2\exp(-u)$, 
\[
\norm{\wh{\Sigma}-\wh{\Sigma}_m}\leq \norm{\wh{\Sigma}}\tau_+(u)
=\norm{\wh{\Sigma}}\left(\frac{(4/3)\wh{r}(u+\log(D))}{m}
    +\sqrt{\frac{2\wh{r}(u+\log(D))}{m}}\right). 
\]

By the triangle inequality, we then have 
\begin{align*}
\norm{\wh{\Sigma}}\leq \norm{\wh{\Sigma}}\tau_+(u)+\norm{\wh{\Sigma}_m}
\qtext{that is}
(1-\tau_+(u))\norm{\wh{\Sigma}}\leq \norm{\wh{\Sigma}_m}
\end{align*}
and
\[
\norm{\wh{\Sigma}_m}\leq \norm{\wh{\Sigma}}\tau_+(u)+\norm{\wh{\Sigma}}
=(1+\tau_+(u))\norm{\wh{\Sigma}}. 
\]
Consequently, as $\wh{r}=1/\norm{\wh{\Sigma}}$, 
\[ 
\left(\frac{1-\tau_+(u)}{1+\tau_+(u)}\right)\wh{r} 
\leq \frac{1-\tau_+(u)}{\norm{\wh{\Sigma}_m}} 
\leq \wh{r}\leq \frac{1+\tau_+(u)}{\norm{\wh{\Sigma}_m}} 
\]
Recalling $\wh{r}\leq D$ always, we choose $m=\alpha D\log(2D/\delta)$ and $u=\log(2/\delta)$ so that 
\[
\tau_+(\log(2/\delta))\leq \frac{(4/3)\wh{r}}{\alpha D}+\sqrt{\frac{2\wh{r}}{\alpha D}}
\leq \frac4{3\alpha}+\sqrt{\frac2{\alpha}}\leq \frac2{3} \qtext{for} \alpha\geq 8.  \qedhere
\]
\end{proof}


\subsection*{Acknowledgements}
This research was performed while the author held an NRC Research Associateship award at the U.S. Air Force Research Laboratory.
I should like to thank Mary, Saint Joseph, and the Holy Trinity for helping me with this work. 

\subsection*{Disclaimers}
The views expressed are those of the author and do not reflect the official guidance or position of the United States Government, the Department of Defense, or of the United States Air Force. 

Statement from the DoD: The appearance of external hyperlinks does not constitute endorsement by the United States Department of Defense (DoD) of the linked websites, or the information, products, or services contained therein. The DoD does not exercise any editorial, security, or other control over the information you may find at these locations.

\section{Appendix}\label{app}
The following lemma shows how to modify the proof of the Hanson-Wright inequality from~\cite{RudelsonVershyninHW2013} (cf.~\cite[chapter~6]{VershyninHDP2018}) to a \qt{bulk} version, looking at the sum of several $\iid$ quadratic forms. 
Note $Z$ here is $Z^\top$ in the main part of the paper. 
Let $Z$ be a $D\times k$ matrix entries $Z_{ij}$ drawn $\iid$ from a mean-zero unit-variance sub-gaussian distribution. 
Write $Z(:,j)$ for the $j$th column of $Z$. 
Let $B$ be a $D\times D$ matrix and
\[
S = \sum_{j=1}^k Z(:,j)^\top B Z(:,j)
\]
Note, with $B=(b_{ij})_{i,j=1}^D$, 
\begin{equation}\label{eqn:ZSumForHW}
Z(:,j)^\top B Z(:,j) = \sum_{q,\,l} Z_{qj} b_{ql} Z_{lj}
\qtext{and}
\E Z(:,j)^\top B Z(:,j) = \sum_q b_{qq} \E Z_{qj}^2, 
\end{equation}
using the mean zero and independence assumptions for the coordinates of $Z(:,j)$, that is for the $Z_{ij}$ when $i$ varies.

\begin{lem}\label{lem:kHansonWright}
Let $S$, $B$, and $Z$ be as above. 
Then for all $t\geq 0$ and either sign choice,
\[
\Prob\set{\pm(S-\E S)\geq kt}
\leq 2\exp\left(-ck\min\left(\frac{t^2}{K^4\norm{B}_F^2},\,\frac{t}{K^2\norm{B}}\right)\right) 
\]
with $c$ an absolute constant (not depending on $Z$) and $K=\norm{Z_{11}}_{\psi_2}$. 
\end{lem}
\begin{rem}
The key point is the additional factor of $k$ in the exponential, compared to the usual Hanson-Wright inequality where $k=1$. 
Here,
\[
\norm{Z}_{\psi_2}=\inf\set{t>0\st \E\exp(Z^2/t)\leq 2}, 
\]
so in particular $\norm{CZ}_{\psi_2}=C\norm{Z}_{\psi_2}$, and $\norm{Z/K}_{\psi_2}=1$. 
If we prove the result for $Z/K$, that is bound 
\[
\Prob\set{\abs{\sum_{j=1}^k \frac{Z(:,j)^\top}{K} B \frac{Z(:,j)}{K}
-\E\sum_{j=1}^k \frac{Z(:,j)^\top}{K} B \frac{Z(:,j)}{K}}\geq kt}, 
\]
then taking $t\mapsto t/K^2$ will give us the bound for the original $Z(:,j)$. 
\end{rem}
\begin{proof}
By equation~\ref{eqn:ZSumForHW}, 
\[
S - \E S= \sum_{j=1}^k \sum_q b_{qq}(Z_{qj}^2 - \E Z_{qj}^2)
+ \sum_{j=1}^k \sum_{q\neq l} b_{ql} Z_{qj} Z_{lj} =: S_1 + S_2. 
\]
By the union bound (Boole's inequality), we can bound the probability 
\[
p:=\Prob\set{S - \E S\geq kt} 
\leq \Prob\set{S_1\geq kt/2} + \Prob\set{S_2\geq kt/2}
=: p_1+p_2, 
\]
for if $S_1<kt/2$ and $S_2<kt/2$, then $S-\E S<kt$. 

We can now use the $\iid$ assumption for the columns, that is, for the $Z_{ij}$ when $j$ varies, 
\begin{align*}
p_1&\leq e^{-kt\lambda_1/2}\E\exp(\lambda_1 S_1)
=\left(e^{-t\lambda_1/2}\E\exp\left(\lambda_1\sum_q b_{qq}(Z_q^2-\E Z_q^2)\right)\right)^k
=\wp_1^k
\end{align*}
and
\begin{align*}
p_2&\leq e^{-kt\lambda_2/2}\E\exp(\lambda_2 S_2)
=\left(e^{-t\lambda_2/2}\E\exp\left(\lambda_2\sum_{q\neq l} b_{ql} Z_q Z_l\right)\right)^k
=\wp_2^k
\end{align*}
The terms $\wp_1$ and $\wp_2$ are the starting points for establishing a proof of the Hanson-Wright inequality~\cite[page~133]{VershyninHDP2018}; the former is for using Bernstein's inequality, while the latter uses decoupling and comparison to the case when $Z$ is a standard Gaussian random vector. 
Consequently, we can use the bounds for $\wp_1$ and $\wp_2$, which both are given by
\[
\max\set{\wp_1,\wp_2}\leq \exp\left(-c\min\left(\frac{t^2}{\norm{B}_F^2},\,\frac{t}{\norm{B}}\right)\right), 
\]
with $c$ an absolute constant (not depending on the distribution of $Z$, as we already rescaled $Z$ to have unit $\psi_2$-norm entries). 
Recalling the $k$th powers, we finally have
\[
p\leq p_1+p_2\leq 2\exp\left(-ck\min\left(\frac{t^2}{\norm{B}_F^2},\,\frac{t}{\norm{B}}\right)\right).\qedhere 
\]
\end{proof}

\begin{lem}\label{lem:HWstableranks}
Let $Z$ be a $k\times D$ random matrix with $\iid$ mean-zero unit-variance sub-gaussian entries. 
Then for a matrix $A$ with columns in $\R^D$,  
\[
\Prob\set{\pm(\norm{ZA}_F^2-k\norm{A}_F^2)\geq k\epsilon\norm{A}_F^2}
\leq 2\exp\left(-ck\min\left(\frac{\epsilon^2r_4(A)}{K^4},\,\frac{\epsilon \,r_\infty(A)}{K^2}\right)\right) 
\]
with $K=\norm{Z_{11}}_{\psi_2}$. 
\end{lem}
\begin{proof}
We use lemma~\ref{lem:kHansonWright}, with $B=AA^\top$, and $Z\mapsto Z^\top$, for then the rows $Z_j$ of $Z$ are written as column vectors, so that
\[
Z_j^\top B Z_j = \norm{A^\top Z_j}_2^2, 
\quad S = \sum_{j=1}^k \norm{A^\top Z_j}_2^2 = \norm{ZA}_F^2, 
\qtext{and}
\E S = k\norm{A}_F^2. 
\]
Using $\norm{B}=\norm{A}^2$, we recover 
\[
\Prob\set{\pm(S-k\norm{A}_F^2)\geq kt}
\leq \exp\left(-ck\min\left(\frac{t^2}{K^4\norm{AA^\top}_F^2},\,\frac{t}{K^2\norm{A}^2}\right)\right). 
\]
Because 
\[
r_\infty(A)
=\frac{\norm{A}_F^2}{\norm{A}^2}
\qtext{and}
r_4(A)
=\frac{\norm{A}_F^4}{\norm{AA^\top}_F^2}, 
\]
the choice $t=\epsilon\norm{A}_F^2$ yields the result.
\end{proof}

If the reader would prefer explicit constants, the following lemma may be convenient, and gives an alternative proof for lemma~\ref{lem:HWstableranks} in the Gaussian case, relying on the explicit moment generating function for the Gaussian distribution.
\begin{lem}\label{lem:DirectBulkHWGaussianCase}
Let $Z$ be a $k\times D$ random matrix with $\iid$ standard Gaussian entries. 
Then for a matrix $A$ with columns in $\R^D$,  
\[
\Prob\set{\norm{ZA}_F^2>(1+\epsilon)k\norm{A}_F^2}
\leq \exp\left(-k\frac\epsilon{8}
    \min\set{\epsilon\,r_4(A),\, r_\infty(A)}\right)
\]
for $\epsilon>0$. 
Also, when $\epsilon\in(0,1)$, 
\[
\Prob\set{\norm{ZA}_F^2>(1+\epsilon)k\norm{A}_F^2}
\leq \exp\left(-k\frac{\epsilon^2}{8}r_\infty(A)\right)
\]
and  
\[
\Prob\set{\norm{ZA}_F^2< (1-\epsilon)k\norm{A}_F^2}
\leq \exp\left(-k\frac{\epsilon^2}{4}r_4(A)\right).  
\]
\end{lem}
Note $r_4(A)\geq r_\infty(A)$ always. 
When $\epsilon\, r_4(A)\geq r_\infty(A)$, there is a savings of one factor of $\epsilon$ in the upper tail; however, for our applications, we do not know the relative sizes of $r_4(A)$ and $r_\infty(A)$, so the $k\epsilon^2r_\infty(A)/8$ bound was included. 
%
\begin{proof}
Let $A=U\Sigma V^\top$ be the SVD of $A$, with $\Sigma=\diag(\vec{\sigma})$ the diagonal matrix of singular values. 
So $A^\top = V\Sigma U^\top$ and by the rotation invariance of standard Gaussian vectors, $A^\top$ acts on a row $Z_i$ of $Z$ as 
\[
A^\top Z_i=V\Sigma U^\top Z_i \sim V\Sigma g_i \qtext{with} g_i\in \R^D
\]
and consequently 
\[
\norm{A^\top g_i}_2^2\sim g_i^\top \Sigma V^\top V\Sigma g_i 
= g_i^\top \Sigma^2 g_i
=\sum_j \sigma_j^2 g_{ij}^2
\]
with the $g_{ij}$ independent standard Gaussian. 

We then have
\[
\norm{ZA}_F^2 = \norm{A^\top Z^\top}_F^2 = \sum_{i=1}^k \norm{A^\top g_i}_2^2 
\]
and for $\lambda>0$ to be determined
\begin{align*}
\Prob\set{\norm{ZA}_F^2>(1+\epsilon)k\norm{A}_F^2}
&\leq e^{-\lambda(1+\epsilon)k\norm{A}_F^2}\E\exp\left(\lambda \sum_{i=1}^k \norm{A^\top g_i}_2^2\right)\\ 
&=\left(e^{-\lambda(1+\epsilon)\norm{A}_F^2}\E \exp\left(\lambda \norm{A^\top g_1}_2^2\right)\right)^k
\end{align*}
with $g_1\in \R^D$ standard Gaussian,  having used the indepdence of the \emph{rows} $\set{g_i}$. 
We can now use the independence of the \emph{columns}, here via the coordinates of $g_1$: 
\begin{align*}
\E\exp(\lambda\norm{A^\top g_1}_2^2)=\prod_j \E\exp(\lambda \sigma_j^2g_{1j}^2)
=\prod_j (1-2\lambda\sigma_j^2)^{-1/2}
\end{align*}
provided $\lambda<1/(2\sigma_1^2)$
via change of variables $y=(1-2\lambda\sigma_j^2)^{1/2}\,x$, 
\begin{align*}
\E\exp(\lambda \sigma_j^2g_{1j}^2)
=\frac1{\sqrt{2\pi}}\int_{-\infty}^\infty \exp\left(\lambda \sigma_j^2 x^2-\frac{x^2}{2}\right)\,dx=(1-2\lambda\sigma_j^2)^{-1/2}. 
\end{align*}
On $[0,1/2]$, the function $x\mapsto e^{x+x^2}(1-x)$ is increasing from $1$, while on $[1/2, 2/3]$ it is decreasing and still greater than 1, as $(10/9)>\log(3)$, so 
\[
\E\exp(\lambda\sigma_j^2 g_{1j}^2)\leq \exp\big(\lambda \sigma_j^2+2\lambda^2\sigma_j^4\big)
\qtext{certainly when} 2\lambda\sigma_1^2\leq  2/3, 
\]
leaving us to minimize
\[
h(\lambda):=-\lambda(1+\epsilon)\norm{A}_F^2+\sum_j (\lambda\sigma_j^2+2\lambda^2\sigma_j^4)
=-\lambda\epsilon\sum_j \sigma_j^2 + 2\lambda^2\sum_j \sigma_j^4.
\]
There will turn out to be two cases. 
If we minimize $h(\lambda)$ directly, the minimizer is  
\[
\lambda^\ast 
    = \frac{\epsilon}{4}\frac{\sum_j \sigma_j^2}{\sum_j \sigma_j^4}
\qtext{at which}
h(\lambda^\ast)
=-\frac{\epsilon^2}{8}\frac{\left(\sum_j \sigma_j^2\right)^2}
    {\sum_j \sigma_j^4}
=-\frac{\epsilon^2}{8} r_4(A). 
\]
This estimate still requires $2\lambda^\ast\sigma_1^2<1$, so if we require $2\lambda^\ast\sigma_1^2\leq 1/2$, we force 
\begin{align*}
\frac1{\epsilon}r_\infty(A)
    =\frac1{\epsilon}\frac{\sum_j \sigma_j^2}{\sigma_1^2}
    >\frac1{\epsilon}4\lambda^\ast\sum_j \sigma_j^2
    =r_4(A). 
\end{align*}
Because we always have $r_4(A)\geq r_\infty(A)$, we can use the $h(\lambda^\ast)$ value when $r_4(A)$ and $r_\infty(A)$ are \qt{comparable} and $\epsilon\in (0,1)$. 

For the other case, when $r_4(A)\geq \epsilon^{-1}r_\infty(A)$, we have
\[
\sum_j \sigma_j^4\leq \epsilon\,\sigma_1^2\sum_j\sigma_j^2
\]
and can upper bound $h(\lambda)$ by $\tilde{h}(\lambda)$ defined as
\[
\tilde{h}(\lambda)
=-\lambda\epsilon\sum_j\sigma_j^2
    +2\lambda^2\alpha\epsilon\,\sigma_1^2\sum_j\sigma_j^2 
\qtext{for any} \alpha\geq 1. 
\]
The minimizer for $\tilde{h}(\lambda)$ is
\[
\tilde{\lambda}^\ast=\frac1{4\alpha\sigma_1^2} 
\qtext{at which}
\tilde{h}(\tilde{\lambda}^\ast)
=-\frac{\epsilon}{\alpha 8}r_\infty(A), 
\]
and this $\tilde{\lambda}^\ast$ also satisfies $2\tilde{\lambda}^\ast\sigma_1^2\leq 1/2<1$ whenever $\alpha\geq 1$. 

When $\epsilon\in (0,1)$, we can also avoid the distinction between the two cases by noting $\sigma_1\geq \sigma_j$ for all $j$, so that 
\[
\sum_j\sigma_j^4\leq \sigma_1^2\sum_j\sigma_j^2
\qtext{which corresponds to taking $\alpha=1/\epsilon$ in the above.} 
\]

For the lower tail, with $\lambda<0$, 
\begin{align*}
\Prob\set{\norm{ZA}_F^2<(1-\epsilon)k\norm{A}_F^2}
\leq e^{-\lambda(1-\epsilon)k\norm{A}_F^2}\E\exp(\lambda \norm{ZA}_F^2)\\
=\left(e^{-\lambda(1-\epsilon)\norm{A}_F^2}
    \prod_j\E\exp(\lambda\sigma_j^2 g_{1j}^2)\right)^k. 
\end{align*}
Because $\lambda<0$, we can estimate the moment generating function in two ways. 
From $e^x\leq 1+x+x^2/2$ for $x\leq 0$, we find
\[
\E\exp(\lambda\sigma_j^2 g_{1j}^2)\leq 1+\lambda\sigma_j^2
    +(3/2)\lambda^2\sigma_j^4
\leq \exp\left(\lambda\sigma_j^2+(3/2)\lambda^2\sigma_j^4\right)
\]
while if we use that $e^{x+x^2/2}(1-x)$ is decreasing to 1 for $x\leq 0$, 
\[
\E\exp(\lambda\sigma_j^2 g_{1j}^2)=(1-2\lambda\sigma_j^2)^{-1/2}
\leq \exp\left(\lambda\sigma_j^2+\lambda^2\sigma_j^4\right). 
\]
Minimizing
\[
h_-(\lambda)=-\lambda(1-\epsilon)\norm{A}_F^2
    +\sum_j(\lambda\sigma_j^2+\beta\lambda^2\sigma_j^4)
    =\lambda\epsilon\norm{A}_F^2+\beta\lambda^2\sum_j\sigma_j^4
\]
yields
\[
h_-(\lambda_-^\ast)=-\frac{\epsilon^2}{4\beta}r_4(A) 
\qtext{at}  
\lambda_-^\ast = -\frac{\epsilon}{2\beta}\frac{\norm{A}_F^2}{\sum_j \sigma_j^4}
\]
with $\beta=3/2$ corresponding to the Taylor expansion bound and $\beta=1$ corresponding to the function bound. 

Putting everything together, and remembering the $k$th power outside, we complete the lemma. 
\end{proof}

The next lemma makes the connection between equation~\eqref{eqn:JLSquaredDistancesPairwise} and equation~\eqref{eqn:JLUnsquaredDistancesPairwise} explicit, and is informed by the form of the target dimension derived from the tail bound rates above. 
In the Gaussian case, $C_+=8$ and $C_-=4$. 
\begin{lem}\label{lem:EpsilonAdjustment}
For $0<\epsilon<1$ and $C_\pm>0$, the requirements 
\[
\frac{C_+}{\wEps_+^2}=\frac{C_-}{\wEps_-^2}, 
\quad
\frac{(1-\epsilon)^2}{\gamma(\epsilon)}=1-\wEps_-
\qtext{and}
\frac{(1+\epsilon)^2}{\gamma(\epsilon)}=1+\wEps_+
\]
have solution $\wEps_+=\theta\,\wEps_-$ with $\theta=\sqrt{C_+/C_-}$, 
\[
1>\wEps_-=\frac{4\epsilon}{(1+\epsilon)^2+\theta(1-\epsilon)^2}, 
\qtext{and}
\gamma(\epsilon)=\frac{(1+\epsilon)^2+\theta(1-\epsilon)^2}{1+\theta}. 
\]
\end{lem}
\begin{proof}
The first equation gives $\wEps_+=\theta\,\wEps_-$. 
Taking $\theta$ times the second equation and adding it to the third gives the equation for $\gamma(\epsilon$. 
Subtracting the second equation from the third yields
\[
(1+\theta)\wEps_-=\frac{(1+\epsilon)^2-(1-\epsilon)^2}{\gamma(\epsilon)}
=\frac{4\epsilon}{\gamma(\epsilon)}. 
\]
Conclude 
\[
\wEps_-=\frac{4\epsilon}{(1+\epsilon)^2+\theta(1-\epsilon)^2}
=\frac{4\epsilon}{(1+\theta)(1-\epsilon)^2+4\epsilon}<1. \qedhere
\]
\end{proof}

\begin{lem}\label{lem:binomEstimate}
For $1\leq j\leq \floor{M/2}$, 
\[
\binom{M}{j} \leq \left(\frac{eM}{j}\right)^j. 
\]
\end{lem}
\begin{proof}
First note $j!\geq (j/e)^j$ by
\[
\frac{j^j}{j!}\leq \sum_{i=0}^\infty \frac{j^i}{i!}=e^j. 
\qtext{We then have}
\binom{M}{j}=\frac{M!}{j!(M-j)!} = \frac1{j!}\prod_{i=0}^{j-1} (M-i)
\leq \frac{M^j}{j!}\leq \left(\frac{eM}{j}\right)^j
\]
from our lower bound for $j!$. 
\end{proof}


\singlespacing
\bibliographystyle{amsalphazentralblatteprint}
\renewcommand\refname{Literature Cited}
\bibliography{Bulk_JL} 

\end{document}